\documentclass[12pt, a4paper]{amsart}
\usepackage{accents,graphicx}
\usepackage{amsmath,amssymb,amsthm,amscd}
\usepackage{adjustbox}
\usepackage{amsmath}
\usepackage{amssymb}
\usepackage{amssymb,amsmath,amsthm,amsfonts,latexsym,euscript}
\newtheorem*{definition}{Definition}
\newtheorem{lemma}{Lemma}
\newtheorem*{corollary}{Corollary}

\newtheorem{theorem}{Theorem}

\usepackage[usenames,dvipsnames]{xcolor}
\usepackage{setspace}
\usepackage{txfonts}
\usepackage{pstricks,pst-plot,pst-node,pst-func}

\newcommand{\vv}{\begin{pspicture}[shift=-.4](-.6,-.5)(.6,.5)
\psbezier(.5;45)(.25;45)(.25;315)(.5;315)
\psbezier(.5;225)(.25;225)(.25;135)(.5;135)
\end{pspicture}}
\newcommand{\hh}{\begin{pspicture}[shift=-.4](-.6,-.5)(.6,.5)
\psbezier(.5;45)(.25;45)(.25;135)(.5;135)
\psbezier(.5;225)(.25;225)(.25;315)(.5;315)
\end{pspicture}}
\newcommand{\btwohvert}{\begin{pspicture}[shift=-.6](-.6,-.7)(.6,.7)
\qline(0, .25)(.35, .6)\qline(0, .25)(-.35, .6)
\qline(0,-.25)(.35,-.6)\qline(0,-.25)(-.35,-.6)
\psline[doubleline=true](0,-.25)(0,.25)
\end{pspicture}}
\newcommand{\btwohhoriz}{\begin{pspicture}[shift=-.4](-.8,-.5)(.8,.5)
\qline( .25,0)( .6,.35)\qline( .25,0)( .6,-.35)
\qline(-.25,0)(-.6,.35)\qline(-.25,0)(-.6,-.35)
\psline[doubleline=true](-.25,0)(.25,0)
\end{pspicture}}

\newcommand{\singleloop}{\begin{pspicture}[shift=-.4](-.6,-.5)(.6,.5)
\pscircle(0,0){.4}
\end{pspicture}}
\newcommand{\doubleloop}{\begin{pspicture}[shift=-.4](-.6,-.5)(.6,.5)
\pscircle[doubleline=true](0,0){.4}
\end{pspicture}}

\newcommand{\btwox}{\begin{pspicture}[shift=-.4](-.6,-.5)(.6,.5)
\qline(.5;45)(.5;225)\qline(.5;135)(.5;315)
\end{pspicture}}

\AtEndDocument{{\footnotesize%
  \textsc{A. Mejia, Department of Mathematics, Bard College, Annandale-On-Hudson, NY., 12504} \par
  \textit{E-mail address}, A.~Mejia: \texttt{am8248@bard.edu} \par
\vspace{3 mm}
\textsc{W. Bloomquist, Department of Mathematics, University of California, Santa Barbara, CA 93106}\par
\textit{E-mail ADDRESS}, W.~Bloomquist: \texttt{bloomquist@math.ucsb.edu} 
 }}

\title{Admissibility and the $C_2$ Spider }
\author[W.Bloomquist]{Wade Bloomquist}
\author[A.Mejia]{Andres Mejia}

\begin{document}

\begin{abstract}
A tensor category is multiplicity-free if for any objects $A,B,C$ we have that $\mathrm{Hom}(A\otimes B\otimes C,\mathbb{C})$ is either $0$ or $1$ dimensional.  It is known that $Rep^{uni}(U_q(\mathfrak{sp}(4)))$ is not multiplicty-free. We find a full subcategory of $Rep^{uni}(U_q(\mathfrak{sp}(4)))$ which is multiplicty-free. A description of the dimension of these $\mathrm{Hom}$ spaces is given for this subcategory, including when $q$ is a root of unity.  The methods used arise from the description, given by Kuperberg, of $Rep^{uni}(U_q(\mathfrak{sp}(4)))$ as a spider.  The main tool is the recursive definition of clasps given by Kim.  In particular, we provide an appropriate notion of admissibility when looking at the $\mathrm{Sp}(4)_k$ ribbon graph invariants with restricted edge labels.
\end{abstract}

\maketitle

\let\thefootnote\relax\footnotetext{The authors were supported by NSF grant DMS-1358884.}

\section{Introduction}
Quantum topology has forged far reaching connections between low dimensional topology and algebra.  As an example, Reshetikhin and Turaev have shown applications of the representation theory of quantum groups towards link invariants, $3$-manifold invariants, and mapping class group representations, through the construction of TQFTs \cite{RT1, RT2}.   

These constructions were reformulated in terms of a skein theoretic approach by Blanchet, Habegger, Masbaum, and Vogel \cite{BHMV}.  In many ways this story follows the rediscovery of Temperley-Lieb algebras, and thus the Jones polynomial, by Jones being formulated in a diagrammatic language by Kauffman \cite{J1,J2,K}.  This diagrammatic interpretation allowed for the development of recoupling theory as described by Kauffman and Lins \cite{KL}.  Recoupling theory uses diagrammatic techniques to perform the computations of the above mentioned link invariants, $3$-manifold invariants, and mapping class group representations through combinatorial means.

The combinatorial spiders of Kuperberg serve, in some sense, as a generalization of the Temperley-Lieb algebra to the Lie algebras of rank $2$ \cite{GK}.  We will only be focusing on the $B_2/C_2$ spider which serves the role of the Temperley-Lieb algebra for $U_q(\mathfrak{sp}(4))$.  We look to develop some necessary results for the recoupling theory associated to this diagrammatic formulation.  In particular we  will utilize the construction of clasps by Kim to recursively make computations \cite{DK}.  

The organization of this paper is as follows.  First a review of the $C_2$ spider is given.  Then it is shown that when restricting only to irreducible representations of highest weight $(p,0)$, that this subcategory is multiplicity-free.  Finally a recursive computation allows for a condition to be found for the behavior of the dimension of these $\mathrm{Hom}$ spaces when $q$ is a root of unity.

\section{{Preliminaries}}
\subsection{What is a Spider?}A spider can be formed out of any pivotal tensor category and a collection of objects.  Rather than give the original definition of Kuperberg, we look to provide a modern formulation.  In particular a spider is a full subcategory whose objects are tensor products of the chosen set of objects and their duals.  In many ways a spider can be thought of as a planar algebra with labeled strands.  In particular if the label set is a single symmetrically self-dual object, meaning it is self dual and it's associated frobenius-schur indicator is $1$, then the associated spider is an unoriented unshaded planar algebra.  This formulation captures the original ideas of a spider being a pivotal tensor category. We will be looking only at the $B_2/C_2$ spider.  By this we mean the spider generated by the fundamental representations  of $Rep^{uni}(U_q(\mathfrak{sp}(4)))$, meaning the representation category of the quantum group $U_q(\mathfrak{sp}(4))$ given unimodal pivotal structure.
\subsection{The Combinatorial $C_2$ Spider}
Kuperberg was able to provide explicit combinatorial constructions for the spiders associated to rank $2$ Lie algebras.  In particular this gives a concrete description for the $\mathrm{Hom}$ spaces of these categories.  Namely, these $\mathrm{Hom}$ spaces are the free vector spaces having a basis of diagrams built out of certain generators subject to local relations.  This is analogous to Temperley-Lieb diagrams, non-crossing planar matchings, forming the basis of $\mathrm{Hom}$ spaces in $Rep^{uni}(U_q(\mathfrak{sl}(2,\mathbb{C})))$.  

We now turn our attention to the combinatorial $C_2$ spider.   As the $C_2$ spider has two generating objects, the two fundamental representations, we have two different strand types in our diagrams.  First $n$ points labeled $1$ and $m$ points labeled $2$ are on the boundary of the unit disk $D^2$. The we obtain basis vectors, called webs, by diagrams generated in the disk by a single element
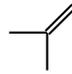
\begin{align*}
\begin{pspicture}(-.7,-.7)(.7,.7)
\psline(-.5,0)(0,0)\psline(0,-.5)(0,0)
\psline[doubleline=true](0,0)(.35,.35)
\end{pspicture}
\end{align*}

subject to the following relations:
\begin{eqnarray}
\begin{split}
\singleloop   =&-\frac{[2][3]}{[6]} \\
\doubleloop  = & \frac{[6][5]}{[3][2]}  \\
\begin{pspicture}(-.6,-.2)(.6,.5)
\psline[doubleline=true](-.5,0)(0,0)
\psbezier(0,0)(.7,.7)(.7,-.7)(0,0)
\end{pspicture}
= & 0 \\
\end{split}
\begin{split}
\begin{pspicture}(-.5,-.2)(.7,.2)
\psline[doubleline=true](-.7,0)(-.3,0)
\pcarc[arcangle=45](-.3,0)(.3,0)
\pcarc[arcangle=-45](-.3,0)(.3,0)
\psline[doubleline=true](.3,0)(.7,0)
\end{pspicture}
 = & -([2])^2 \begin{pspicture}(-.9,-.2)(.6,.5)
\psline[doubleline=true](-.5,0)(.5,0)
\end{pspicture} \\
\begin{pspicture}(-.9,-.1)(.9,.5)
\psline[doubleline=true](.3;90)(.7;90)
\psline[doubleline=true](.3;210)(.7;210)
\psline[doubleline=true](.3;330)(.7;330)
\pcarc[arcangle=-15](.3;90)(.3;210)
\pcarc[arcangle=-15](.3;210)(.3;330)
\pcarc[arcangle=-15](.3;330)(.3;90)
\end{pspicture}   =  &  0\\
\vspace{2 mm}
\btwohvert  - & \btwohhoriz = \hh - \vv \label{fb2elliptic}
\end{split}
\end{eqnarray}

We will say that single strands are of type $1$, meaning they meet a $1$ at the boundary, and double strands are of type $2$, meaning they meet a $2$ at the boundary, using the notation $D_{n,m}$ to represent the disk with $n$ points on the boundary that attach to single strands, and $m$ points that attach to double strands.  We will then omit the label at the boundary as the strand type will make it clear. Here is an example of $D_{4,0}$:
\begin{align*}
\begin{pspicture}(-1,-1)(1,1)
\pscircle(0,0){1}
\psbezier(1;300)(.5;300)(.5;60)(1;60)
\psbezier(1;120)(.5;120)(.5;240)(1;240)
\end{pspicture}
\hspace{1cm}
\begin{pspicture}(-1,-1)(1,1)
\pscircle(0,0){1}
\psbezier(1;60)(.5;60)(.5;120)(1;120)
\psbezier(1;240)(.5;240)(.5;300)(1;300)
\end{pspicture}
\hspace{1cm}
{\begin{pspicture}(-1,-1)(1,1)
\pscircle(0,0){1}
\qline(0, .25)(.5, .866)\qline(0, .25)(-.5, .866)
\qline(0,-.25)(.5,-.866)\qline(0,-.25)(-.5,-.866)
\psline[doubleline=true](0,-.25)(0,.25)
\end{pspicture}}
\end{align*}

Additionally, by introducing the following notation we will be able to describe our diagram without any internal double edges:
$$
\btwox = \btwohhoriz - \vv
$$

\subsection{Clasps}

\begin{definition} A \textit{cut path} is a path whose endpoints separate the web space into two disjoint parts.  A cut path is said to be \textit{minimal} if it crosses as few strands as possible. If a cut path crosses $n$ single strands, $k$ double strands, and $k^{\prime}$ tetravalent vertices, it has weight $n\lambda_1+(k+k^{\prime})\lambda_2.$
\end{definition}

We further subject weights to  the following partial ordering:
\begin{align*} a \lambda_1+b \lambda_2 &\succ (a-2)\lambda_1+(b+1)\lambda_2\\
a \lambda_1+b \lambda_2 &\succ(a+2)\lambda_1+(b-2)\lambda_2
\end{align*}

\begin{definition}A clasp of weight $n \lambda_1+k\lambda_2$ , denoted $P_{n,k}$, 
is an idempotent  consisting of $n$ type $1$ strands and $k$ type $2$ strands that annihilate a web space if there exists a cut path of weight less than its own.
\end{definition}

Kuperberg showed that clasps are unique, while Kim provided a recursive construction for clasps of type $P_{m,0}$ or $P_{0,m}$ for the $C_2$ Spider (shown in figures $1$ and $2$.) We will say that clasps have the \textit{cut path property}.

\begin{figure}
\caption{The $P_{n,0}$ Clasp expansion}
\begin{align*}
\begin{pspicture}[shift=-.7](-.2,-.3)(1.2,1.3) \rput[t](.5,-.1){$n$}
\qline(.5,0)(.5,.4) \qline(.5,.6)(.5,1) \rput[b](.5,1.1){$n$}
\psframe[linecolor=black](0,.4)(1,.6)
\end{pspicture}
= \begin{pspicture}[shift=-.7](-.2,-.3)(1.4,1.3) \rput[t](.5,-0.1){$n-1$}
\qline(.5,0)(.5,.4) \qline(.5,.6)(.5,1) \rput[b](.5,1.1){$n-1$}
\qline(1.3,0)(1.3,1) \psframe[linecolor=black](0,.4)(1,.6)
\end{pspicture}
+ \frac{[2n][n+1][n-1]}{[2n+2][n][n]}
\begin{pspicture}[shift=-1.2](-.2,-.3)(1.45,2.3) \rput[t](.5,-.1){$n-1$}
\qline(.5,0)(.5,.4) \psframe[linecolor=black](0,.4)(1,.6)
\qline(.25,.6)(.25,1.4) \rput[l](.35,1){$n-2$} \qline(.5,1.6)(.5,2)
\rput[b](.5,2.1){$n-1$} \psarc(1,.6){.2}{0}{180}
\qline(1.2,0)(1.2,.6) \psarc(1,1.4){.2}{180}{0}
\qline(1.2,1.4)(1.2,2) \psframe[linecolor=black](0,1.4)(1,1.6)
\end{pspicture}
+ \frac{[n-1]}{[n][2]} \begin{pspicture}[shift=-1.2](-.2,-.3)(1.45,2.3)
\rput[t](.5,-.1){$n-1$} \qline(.5,0)(.5,.4) \qline(.25,.6)(.25,1.4)
\qline(.5,1.6)(.5,2) \qline(1.2,1.4)(1.2,2) \rput[b](.5,2.1){$n-1$}
\qline(.8,.6)(1.2,1.4) \qline(1.2,0)(1.2,.6) \qline(.8,1.4)(1.2,.6)
\psframe[linecolor=black](0,.4)(1,.6)
\psframe[linecolor=black](0,1.4)(1,1.6)
\end{pspicture}
\end{align*}
\end{figure}
\begin{figure}
\caption{The $P_{0,n}$ clasp expansion}
\begin{align*}
\begin{pspicture}[shift=-.7](-.2,-.3)(1.2,1.3)
\rput[t](.5,-.1){$n$}\psline[doubleline=true](.5,0)(.5,.4)
\psline[doubleline=true](.5,.6)(.5,1)
\psframe[linecolor=black](0,.4)(1,.6) \rput[b](.5,1.1){$n$}
\end{pspicture}
= \begin{pspicture}[shift=-.7](-.2,-.3)(1.4,1.3)
\rput[t](.5,-0.1){$n-1$}\psline[doubleline=true](.5,0)(.5,.4)
\psline[doubleline=true](.5,.6)(.5,1)\rput[b](.5,1.1){$n-1$}
\psline[doubleline=true](1.3,0)(1.3,1)
\psframe[linecolor=black](0,.4)(1,.6)
\end{pspicture}
+ \frac{[2n-1][2n-2]}{[2n+1][2]} \begin{pspicture}[shift=-1.2](-.2,-.3)(1.45,2.3)
\rput[t](.5,-.1){$n-1$} \psline[doubleline=true](.5,0)(.5,.4)
\psline[doubleline=true](.25,.6)(.25,1.4)\rput[l](.35,1){$n-2$}
\psline[doubleline=true](.5,1.6)(.5,2) \rput[b](.5,2.1){$n-1$}
\psarc[doubleline=true](1,.6){.2}{0}{180}
\psline[doubleline=true](1.2,0)(1.2,.6)
\psarc[doubleline=true](1,1.4){.2}{180}{0}
\psline[doubleline=true](1.2,1.4)(1.2,2)
\psframe[linecolor=black](0,.4)(1,.6)
\psframe[linecolor=black](0,1.4)(1,1.6)
\end{pspicture}
+ \frac{[2n-2]}{[2n][2][2]}\begin{pspicture}[shift=-1.2](-.2,-.3)(1.45,2.3)
\rput[t](.5,-.1){$n-1$} \psline[doubleline=true](.5,0)(.5,.4)
\psline[doubleline=true](.25,.6)(.25,1.4)
\psline[doubleline=true](.5,1.6)(.5,2) \rput[b](.5,2.1){$n-1$}
\psline[doubleline=true,nodesep=1pt](.8,.6)(.97,.94)
\psline[doubleline=true,nodesep=1pt](1.04,1.06)(1.2,1.4)
\psline[doubleline=true](1.2,0)(1.2,.6)
\psline[doubleline=true,nodesep=1pt](.8,1.4)(.97,1.06)
\psline[doubleline=true,nodesep=1pt](1.03,.94)(1.2,.6)
\psline[doubleline=true](1.2,1.4)(1.2,2)
\qline(.942,.94)(.942,1.06)\qline(1.06,.94)(1.06,1.06)
\psframe[linecolor=black](0,.4)(1,.6)
\psframe[linecolor=black](0,1.4)(1,1.6)
\end{pspicture}
\end{align*}
\end{figure}

One can also form a \textit{clasped web space}, obtained by requiring that each boundary strand be attached to a clasp (on the boundary.) From this, it is clear that any cut paths of lower weight than the sum of weights of each clasp will annihilate the diagram. 

Of particular importance are the so-called fusion, or triple-clasped, spaces which are composed of diagrams with exactly three clasps attached to the boundary. We denote these fusion spaces by 
\[I((p_1,q_1),(p_2,q_2),(p_3,q_3)),\] 
where the three clasps attached to the boundary are $P_{(p_1,q_1)}, P_{(p_2,q_2)},$ and $P_{(p_3,q_3)}$.  The importance of these spaces is apparent as they are isomorphic to 
\[ \mathrm{Hom}(V_{(p_1,q_1)}\otimes V_{(p_2,q_2)}\otimes V_{(p_3,q_3)},\mathbb{C}),\]
using the equivalence proven by Kuperberg.

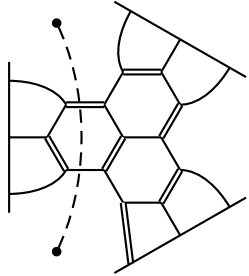
\begin{figure}
$$
\begin{pspicture}(-2,-2)(2,2)
\pnode(0,.866){a1}\pnode(.5,.866){a2}
\pnode(-.75,.433){a3}\pnode(-.25,.433){a4}\pnode(.75,.433){a5}
\pnode(-1,0){a6}\pnode(0,0){a7}\pnode(.5,0){a8}
\pnode(-.75,-.433){a9}\pnode(-.25,-.433){a10}\pnode( .75,-.433){a11}
\pnode(0,-.866){a12}\pnode(.5,-.866){a13}
\pnode(1.5; 60){b2}\pnode([nodesep=.75,angle=150]b2){b3}
\pnode([nodesep=.75,angle=330]b2){b1}
\pnode(1.5;180){b5}\pnode([nodesep=.75,angle=270]b5){b6}
\pnode([nodesep=.75,angle= 90]b5){b4}
\pnode(1.5;300){b8}\pnode([nodesep=.75,angle= 30]b8){b9}
\pnode([nodesep=.75,angle=210]b8){b7}
\psline[doubleline=true,nodesep=1pt](a1)(a2)
\ncline{a1}{a4}\ncline{a2}{a5}
\psline[doubleline=true,nodesep=1pt](a3)(a4)\ncline{a3}{a6}\ncline{a4}{a7}
\psline[doubleline=true, nodesep=1pt](a5)(a8)
\psline[doubleline=true, nodesep=1pt](a6)(a9)
\ncline{a7}{a8}
\psline[doubleline=true, nodesep=1pt](a7)(a10)\ncline{a8}{a11}
\ncline{a9}{a10}\ncline{a10}{a12}
\psline[doubleline=true, nodesep=1pt](a11)(a13)\ncline{a12}{a13}
\ncline{a2}{b2}\nccurve[angleA=120,angleB=240]{a1}{b3}
\nccurve[angleA=0,angleB=240]{a5}{b1}
\ncline{a6}{b5}\nccurve[angleA=120,angleB=0]{a3}{b4}
\nccurve[angleA=240,angleB=0]{a9}{b6}
\ncline{a13}{b8}\psline[doubleline=true, nodesep=1pt](a12)(b7)
\nccurve[angleA=0,angleB=120]{a11}{b9}
\psline([nodesep=.25,angle=150]b3)([nodesep=.25,angle=330]b1)
\psline([nodesep=.25,angle=270]b6)([nodesep=.25,angle= 90]b4)
\psline([nodesep=.25,angle= 30]b9)([nodesep=.25,angle=210]b7)
\rput([nodesep=.2,angle= 60]b1){}
\rput([nodesep=.2,angle= 60]b2){}
\rput([nodesep=.2,angle= 60]b3){}
\rput([nodesep=.2,angle=180]b4){}
\rput([nodesep=.2,angle=180]b5){}
\rput([nodesep=.2,angle=180]b6){}
\rput([nodesep=.2,angle=300]b7){}
\rput([nodesep=.2,angle=300]b8){}
\rput([nodesep=.2,angle=300]b9){}
\pcarc[arrows=*-*,linestyle=dashed,arcangle=25]
    (1.75;120)(1.75;240)
\end{pspicture}
$$
\caption{A triple clasped null diagram}
\end{figure}

\subsection{Towards a Fusion Category}
Up until now the majority of our discussion has been independent of whether $q$ is generic or a root of unity.  In particular, the spiders described above have only been described as pivotal tensor categories.  The finitely many objects needed to be a fusion category provides a main difference from the general pivotal tensor categories described above. Thus some process will be required to drop to finitely many simple objects.  This is done by a semi-simplification through modding out by negligible morphisms.  In the language of spiders this corresponds to modding out the morphisms which have trace $0$.  It is seen that when $q$ is a root of unity of order $2(2k+6)$ the clasps of weight $(p,q)$ where $p+q>k$ are negligible.
The details of this process were first worked out by Turaev and Wenzl \cite{TW}.  They constructed semi-simple and modular categories using the $\mathfrak{sp}(4)$ link invariant.  Following this work, Blanchet and Beliakova extended these results and showed the connections to the modular categories coming from quantum groups \cite{BB}.  When $q$ is a root of unity of order $2(2k+6)$ we will call this fusion category $\mathrm{Sp}(4)_k$.   

\subsection{Theta Nets}

From triple-clasped spaces, we can associate a closed web, called a theta net.  We look at the constant, $\theta((p_1,q_1),(p_2,q_2),(p_3,q_3),v,w)$, evaluated by looking at composing 
\[v\in \mathrm{Hom}(\mathbb{C}, V_{(p_1,q_1)}\otimes V_{(p_2,q_2)}\otimes V_{(p_3,q_3)})\]
with 
\[w\in \mathrm{Hom}(V_{(p_1,q_1)}\otimes V_{(p_2,q_2)}\otimes V_{(p_3,q_3)},\mathbb{C}).\]
This corresponds to the diagram $v$ on the left connected to diagram $w$ on the right.Of particular interest is the theta net when $w=v$, since when this net is $0$, we have that this morphism is negligible, and thus when looking at $\mathrm{Sp}(4)_k$, that basis vector is lost in the semi-simplification process.

The main goal of this paper is to establish the value for 
\[\theta((a,0),(b,0),(c,0),v,v),\]
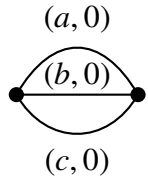
$$\begin{pspicture}[shift=-1.6](-.95,-1.7)(.95,1.7) \pscircle*(-.8,0){.1}
\pscircle*(.8,0){.1}
\pccurve[angleA=60,angleB=120,ncurv=1.2](-.8,0)(.8,0)
\pccurve[angleA=0,angleB=180,ncurv=1.2](-.8,0)(.8,0)
\pccurve[angleA=-60,angleB=-120,ncurv=1](-.8,0)(.8,0)
\rput(0,1){$(a,0)$}\rput(0,.25){$(b,0)$}\rput(0,-.89){$(c,0)$}
\end{pspicture}
$$
as this will allow us to describe $\mathrm{dim}(\mathrm{Hom}(\mathbb{C}, V_{(a,0)}\otimes V_{(b,0)}\otimes V_{(c,0)}))$ in $\mathrm{Sp}(4)_k$.  In fact for the rest of the paper we will only be examining clasps of type $(p,0)$ and so we will often omit $(p,0)$ and simply write $p$.
\section{Computing Theta Nets}

We now provide a quick verification that closed webs behave as we know they will in the combinatorial $C_2$ spider setting.

\begin{lemma} Any closed web $D_{p,q}$ resolves to a constant.
\end{lemma}
\begin{proof} We first change basis, so that we need only consider tetravalent vertices. We proceed by induction on the number of faces in $S$. First, it is clear that if $S$ contains only one face, it reduces to $-\frac{[6][2]}{[3]}$. We proceed with  an Euler characteristic argument. Now assume that we have $E$ faces. Let $P_k$ denote the number of faces with $k$ edges in $S$, $E$ the number of edges, and $V$ the number of vertices. We have that $2E=4V$, and summing over the faces we find that The following formulas are immediate: $ 2E=\sum_{k} k P_k$.  Since the Euler characteristic for a planar graph is $1$, we see that $4V-4E+4F=4$, implying by substitution that
\begin{align*}\sum_{k}(4-k)P_k=4.
\end{align*}
However, since for $k \geq 4$, the left hand side is negative, we know there must have existed a digon or triangle in $S$, and by our relations, these resolve down, reducing the number of faces implying that the inductive hypothesis applies.

\end{proof}
This tells us that closed webs, for example example theta nets, can be evaluated to constants solely using the combinatorial framework of the $C_2$ spider.

The following lemma tells us that for generic $q$ the triple clasped space 
\[I((a,0),(b,0),(c,0))\]
has the same admissibility conditions as the $A_1$ spider.  Namely, $a+b+c$ is even and all three triangle inequalities are satisfied. These arise from showing that any diagrams with a tetravalent vertex are annihilated, so the admissibility conditions have a simple combinatorial interpretation.

\begin{lemma} The dimension of a labeled triple-clasped space with $n$ single strands and $0$ double strands, is either $0$ or $1$
\end{lemma}
\begin{proof} Again, we change basis to tetravalent vertices to eliminate internal type $2$ strands.  Suppose we have the labeling $(a,b,c)$, and let $\mathbf{a,b,c}$ denote the three clasps respectively. We will induct on the number of faces in our diagram. Assuming that there are no faces, we see that this condition is equivalent to assuming that there is at most one tetravalent vertex. If there are no tetravalent vertices, we are done since this establishes the claim. If there is a tetravalent vertex, we have the following picture:
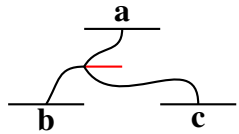
\begin{align*}
\begin{pspicture}(-2,-1.2)(2,.5)
\psline(-1,-1)(0,-1)
\psline(1,-1)(2,-1)
\psline(0,0)(1,0)
\rput(.5,.2){$\bf{a}$}
\rput(-.5,-1.2){$\bf{b}$}
\rput(1.5,-1.2){$\bf{c}$}
\pccurve[angleA=60,angleB=180,ncurv=1](-.5,-1)(0,-.5)
\psline[linecolor=red](0,-.5)(.5,-.5)
\pccurve[angleA=-60,angleB=90,ncurv=1](0,-.5)(1.5,-1)
\pccurve[angleA=60,angleB=270,ncurv=1](0,-.5)(.5,0)
\end{pspicture}
\end{align*}
so we have a cut path by a pigeonhole argument. Hence, we can assume that all three edges attach to $\mathbf{c}$. This also creates a cut path, so we are done.

\vspace{2 mm}

Now, assume that the claim holds for all labellings with $k-1$ faces. First, we choose a tetravalent vertex in lowest position (our diagram should at least be isotopic to one with a vertex in lowest position). We can assume since the diagram is not null that no two strands attach to either $b$ or $c$  Directly above the vertex, we introduce a "cut," and get the following picture:

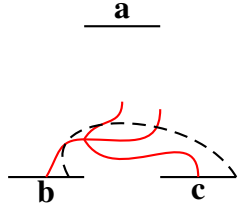
\begin{align*}
\begin{pspicture}(-2,-1.2)(2,.5)
\psline(-1,-1)(0,-1)
\psline(1,-1)(2,-1)
\psline(0,1)(1,1)
\rput(.5,1.2){$\bf{a}$}
\rput(-.5,-1.2){$\bf{b}$}
\rput(1.5,-1.2){$\bf{c}$}
\pccurve[angleA=60,angleB=180,ncurv=1, linecolor=red](-.5,-1)(0,-.5)
\pccurve[angleA=0,angleB=270,ncurv=1, linecolor=red](0,-.5)(1,-.1)
\pccurve[angleA=-60,angleB=90,ncurv=1,linecolor=red](0,-.5)(1.5,-1)
\pccurve[angleA=60,angleB=270,ncurv=1,linecolor=red](0,-.5)(.5,0)
\pccurve[angleA=120,angleB=120,ncurv=1,linestyle=dashed](-.2,-1)(2,-1)
\end{pspicture}
\end{align*}
We consider the resulting diagram, where there are two cases: if the vertex bounded a face, the cut reduces the number of faces, and the inductive hypothesis applies. Otherwise, there are three edges, and two remaining clasps that the edges must connect to, implying that there is a cut path by the previous pigeonhole argument.
\end{proof}
This result tells us that $\mathrm{Hom}(V_a\otimes V_b\otimes V_c,\mathbb{C})$ is either $0$ or $1$ dimensional and so we will use the notation
\[\theta(a,b,c,v,v):=\theta(a,b,c).\]
In the language of ribbon graph invariants, this allows us to leave vertices uncolored when restricting edges to the labels above.

We will denote the RHS in Lemma $3$ by $Net(m,n,p)$.
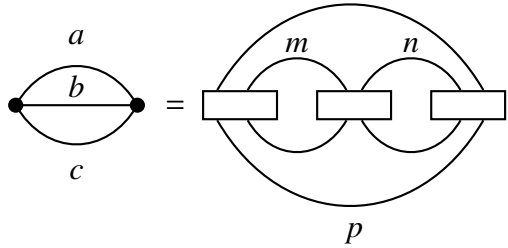
\begin{lemma}
$$\begin{pspicture}[shift=-1.6](-.95,-1.7)(.95,1.7) \pscircle*(-.8,0){.1}
\pscircle*(.8,0){.1}
\pccurve[angleA=60,angleB=120,ncurv=1](-.8,0)(.8,0)
\pccurve[angleA=0,angleB=180,ncurv=1](-.8,0)(.8,0)
\pccurve[angleA=-60,angleB=-120,ncurv=1](-.8,0)(.8,0)
\rput(0,.9){$a$}\rput(0,.25){$b$}\rput(0,-.89){$c$}
\end{pspicture}\hspace{1 mm}
=
\hspace{1 mm}
\begin{pspicture}[shift=-.9](-.5,-1)(4,1)
\psframe(-.5,-.2)(.5,.2)
\psframe(1,-.2)(2,.2)
\psframe(2.5,-.2)(3.5,.2)
\pccurve[angleA=60,angleB=120,ncurv=1](.1,.2)(1.4,.2)
\pccurve[angleA=-60,angleB=-120,ncurv=1](.1,-.2)(1.4,-.2)
\pccurve[angleA=60,angleB=120,ncurv=1](1.6,.2)(2.9,.2)
\pccurve[angleA=-60,angleB=-120,ncurv=1](1.6,-.2)(2.9,-.2)
\pccurve[angleA=60,angleB=120,ncurv=1](-.3,.2)(3.2,.2)
\pccurve[angleA=-60,angleB=-120,ncurv=1](-.3,-.2)(3.2,-.2)
\rput(.75,.8){$m$} \rput(2.25,.8){$n$} \rput(1.5,-1.7){$p$}
\end{pspicture}
$$
where \begin{align*}
m=\frac{a+b-c}{2}&& n=\frac{b+c-a}{2} && p=\frac{a+c-b}{2}
\end{align*}
\end{lemma}
\begin{proof}
This is an immediate consequence of Lemma $2$ and a $90^\circ$ rotation.
\end{proof}
This lemma implies
\[\theta(a,b,c)=Net(m,n,p).\]

\begin{lemma} 
\begin{align*}\mathrm{Tr}(P_{p,0})=\left(\frac{[2p+4]}{[4]}\right)\left(\frac{[3+p][p+1]}{[3]}\right).
\end{align*}
\end{lemma}

\begin{proof}
We proceed by induction, using the recursive definition given by Kim \cite{DK}.

The base case is clear, since the trace of $P_1$ is nothing but a loop that evaluates to $\frac{-[6][2]}{[3]}$ which agrees with the formula above. This can be calculated  as follows, using Kim's double clasp expansion and taking the trace:
$$\begin{pspicture}[shift=-.7](-.2,-.3)(1.2,1.3) \rput[t](.5,-.1){$n$}
\qline(.5,0)(.5,.4) \qline(.5,.6)(.5,1) \rput[b](.5,1.1){$n$}
\psframe[linecolor=black](0,.4)(1,.6)
\pccurve[angleA=180,angleB=180,ncurv=1.5](.5,1)(.5,0)
\end{pspicture}
= \begin{pspicture}[shift=-.7](-.2,-.3)(1.7,1.3) \rput[t](.5,-0.1){$n-1$}
\qline(.5,0)(.5,.4) \qline(.5,.6)(.5,1)
\pccurve[angleA=180,angleB=180,ncurv=1.5](.5,1)(.5,0)
\rput[b](.5,1.1){$n-1$}
\qline(1.3,0)(1.3,1) \psframe[linecolor=black](0,.4)(1,.6)
\pnode(1.3,0){a1}
\pnode(1.3,1){a2}
\pccurve(a1)(a2)
\end{pspicture}
+ \frac{[2n][n+1][n-1]}{[2n+2][n][n]}
\begin{pspicture}[shift=-1.2](-.2,-.3)(1.9,2.3)
\rput[t](.5,-.1){$n-1$}
\qline(.5,0)(.5,.4) \psframe[linecolor=black](0,.4)(1,.6)
\qline(.25,.6)(.25,1.4) \rput[l](.35,1){$n-2$} \qline(.5,1.6)(.5,2)
\rput[b](.5,2.1){$n-1$} \psarc(1,.6){.2}{0}{180}
\qline(1.2,0)(1.2,.6) 
\psarc(1,1.4){.2}{180}{0}
\qline(1.2,1.4)(1.2,2)
\psframe[linecolor=black](0,1.4)(1,1.6)
\pnode(1.2,0){a1}
\pnode(1.2,2){a2}
\pccurve(a1)(a2)
\pccurve[angleA=180,angleB=180,ncurv=.9](.5,2)(.5,0)
\end{pspicture}
+ \frac{[n-1]}{[n][2]} \begin{pspicture}[shift=-1.2](-.2,-.3)(1.45,2.3)
\rput[t](.5,-.1){$n-1$} \qline(.5,0)(.5,.4) \qline(.25,.6)(.25,1.4)
\qline(.5,1.6)(.5,2)  \rput[b](.5,2.1){$n-1$}
\qline(.8,.6)(1.2,1.4)  \qline(.8,1.4)(1.2,.6)
\psframe[linecolor=black](0,.4)(1,.6)
\psframe[linecolor=black](0,1.4)(1,1.6)
\pnode(1.2,1.4){a1}
\pnode(1.2,.6){a2}
\pccurve(a1)(a2)
\pccurve[angleA=180,angleB=180,ncurv=.9](.5,2)(.5,0)
\end{pspicture}
$$
where taking the trace amounts to the trace of the $P_{n-1,0}$ along with some factors: we resolve the first summand by multiplying by the loop constant $-\frac{[6][2]}{3}$; we resolve the second by using idempotence, so it is merely $P_{n-1,0}$; we resolve the third by multiplying $\frac{[6][2]}{3}$ (changing basis again, we see that one summand dies, and the second subtracts off a loop constant.)

From this, we obtain that
\begin{align*}P_n&=\mathrm{Tr}(P_{n-1}) \left(\frac{-[6][2]}{[3]}+\frac{[2n][n+1][n-1]}{[2n+2][n]}+\frac{[n-1][6][2]}{[n][2][3]}\right)\\
&=\left(\frac{[4+n][n]}{[3]}\cdot \frac{[2n+2]}{[3]}\right) \cdot \left(\frac{-[6][2]}{[3]}+\frac{[2n][n+1][n-1]}{[2n+2][n]^2}+\frac{[n-1][6][2]}{[n][2][3]}\right)\\
&=\frac{[2n+4]}{[4]}\cdot\frac{[3+n][n+1]}{[3]},
\end{align*}
as desired.
\end{proof}
\begin{theorem} $Net(m,n,0)=Tr(P_{m+n})$
\end{theorem}
\begin{proof}
\begin{align*}
\begin{pspicture}[shift=-.9](-.5,-1)(3.5,1)
\psframe(-.5,-.2)(.5,.2)
\psframe(1,-.2)(2,.2)
\psframe(2.5,-.2)(3.5,.2)
\pccurve[angleA=60,angleB=120,ncurv=1](.1,.2)(1.4,.2)
\pccurve[angleA=-60,angleB=-120,ncurv=1](.1,-.2)(1.4,-.2)
\pccurve[angleA=60,angleB=120,ncurv=1](1.6,.2)(2.9,.2)
\pccurve[angleA=-60,angleB=-120,ncurv=1](1.6,-.2)(2.9,-.2)
\rput(.75,.8){$m$} \rput(2.25,.8){$n$} 
\end{pspicture}
=
\begin{pspicture}[shift=-1.5](0,-2)(4,1)
\psframe(.3,-.9)(1.3,-.5)
\psframe(.5,-.2)(2.5,.2)
\psframe(1.7,-.9)(2.7,-.5)
\psline(.8,-.2)(.8,-.5)
\qline(.8,.2)(.8,.4)
\psarc(.5,.4){.3}{0}{180}
\qline(.2,.4)(.2,-1.1)
\psarc(.5,-1.1){.3}{180}{0}
\qline(.8,-1.1)(.8,-.9)
\psline(2.2,-.2)(2.2,-.5)
\qline(2.2,.2)(2.2,.4)
\qline(2.2,-.9)(2.2,-1.1)
\psarc(2.5,.4){.3}{0}{180}
\psarc(2.5,-1.1){.3}{180}{0}
\qline(2.8,-1.1)(2.8,.4)
\rput(.75,.8){$m$} \rput(2.25,.8){$n$} 
\end{pspicture}
=
\begin{pspicture}[shift=-1.7](0,-2)(3,1)
\psframe(.5,-.2)(2.5,.2)
\psline(.8,-.2)(.8,-.4)
\qline(.8,.2)(.8,.4)
\psarc(.5,.4){.3}{0}{180}
\qline(.2,.4)(.2,-.4)
\psarc(.5,-.4){.3}{180}{0}
\psline(2.2,-.2)(2.2,-.4)
\qline(2.2,.2)(2.2,.4)
\psarc(2.5,.4){.3}{0}{180}
\psarc(2.5,-.4){.3}{180}{0}
\qline(2.8,-.4)(2.8,.4)
\rput(.75,.8){$m$} \rput(2.25,.8){$n$} 
\end{pspicture}
=\begin{pspicture}[shift=-1.7](0,-2)(3,1)
\psframe(.5,-.2)(2.5,.2)
\psline(2.2,-.2)(2.2,-.4)
\qline(2.2,.2)(2.2,.4)
\psarc(2.5,.4){.3}{0}{180}
\psarc(2.5,-.4){.3}{180}{0}
\qline(2.8,-.4)(2.8,.4)
\rput(2.25,.8){$m+n$}
\end{pspicture}
\end{align*}
which is exactly the trace.
\end{proof}

We will abbreviate the expansion coefficients for $P_n$ by defining
\begin{align*}
\alpha_n:=\frac{[2n][n+1][n-1]}{[2n+2][n]^2} && \beta_n:=\frac{[n-1]}{[n][2]},
\end{align*}
and for further convenience, we will define
\begin{align*}
A_i:=-\frac{[6][2]}{[3]}+\alpha_{m+i}+\alpha_{n+i}+\frac{[6][2]}{[3]}(\beta_{n+i}+\beta_{m+i})-[4][2]\beta_{n+i}\cdot \beta_{m+i}&& B_i:=\alpha_{n+i} \cdot \alpha_{m+i}
\end{align*}
definitions that will be made clear by the next few lemmas. The first step of our recursion is easy:

\begin{lemma} $Net(m,n,1)=A_1\cdot Net(m,n,0)$
\end{lemma}
\begin{proof} Using the double clasp expansion we obtain the equation

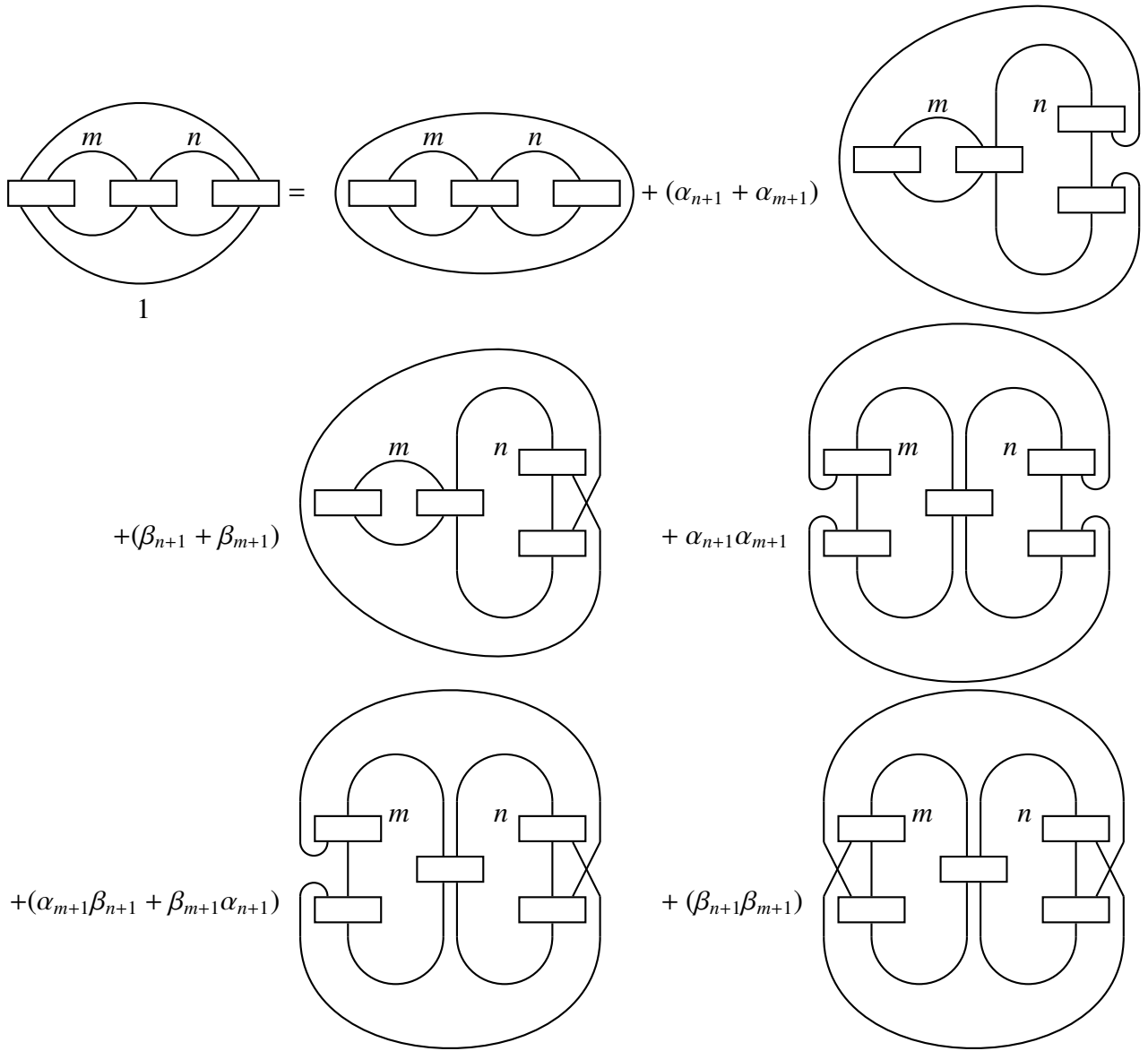
\begin{align*}
\begin{pspicture}[shift=-.9](-.5,-1)(3.5,1)
\psframe(-.5,-.2)(.5,.2)
\psframe(1,-.2)(2,.2)
\psframe(2.5,-.2)(3.5,.2)
\pccurve[angleA=60,angleB=120,ncurv=1](.1,.2)(1.4,.2)
\pccurve[angleA=-60,angleB=-120,ncurv=1](.1,-.2)(1.4,-.2)
\pccurve[angleA=60,angleB=120,ncurv=1](1.6,.2)(2.9,.2)
\pccurve[angleA=-60,angleB=-120,ncurv=1](1.6,-.2)(2.9,-.2)
\pccurve[angleA=60,angleB=120,ncurv=1](-.3,.2)(3.2,.2)
\pccurve[angleA=-60,angleB=-120,ncurv=1](-.3,-.2)(3.2,-.2)
\rput(.75,.8){$m$} \rput(2.25,.8){$n$} \rput(1.5,-1.7){$1$}
\end{pspicture}
&=
\begin{pspicture}[shift=-.9](-1,-1)(3.7,1.2)
\psframe(-.5,-.2)(.5,.2)
\psframe(1,-.2)(2,.2)
\psframe(2.5,-.2)(3.5,.2)
\pccurve[angleA=60,angleB=120,ncurv=1](.1,.2)(1.4,.2)
\pccurve[angleA=-60,angleB=-120,ncurv=1](.1,-.2)(1.4,-.2)
\pccurve[angleA=60,angleB=120,ncurv=1](1.6,.2)(2.9,.2)
\pccurve[angleA=-60,angleB=-120,ncurv=1](1.6,-.2)(2.9,-.2)
\rput(.75,.8){$m$} \rput(2.25,.8){$n$} 
\pnode(1.5,0){a1}
\psellipse(1.5,0)(2.2,1.2)
\end{pspicture}
+(\alpha_{n+1}+\alpha_{m+1})
\begin{pspicture}[shift=-1.4](-1,-2)(4.5,2)
\psframe(-.5,-.2)(.5,.2)
\psframe(1,-.2)(2,.2)
\pccurve[angleA=60,angleB=120,ncurv=1](.1,.2)(1.4,.2)
\pccurve[angleA=-60,angleB=-120,ncurv=1](.1,-.2)(1.4,-.2)
\rput(.75,.8){$m$} \rput(2.25,.8){$n$} 
\qline(3,-1)(3,-.8) \psframe[linecolor=black](2.5,-.8)(3.5,-.4)
\qline(3,-.4)(3,.4)  \qline(3,.8)(3,1)
 \psarc(3.5,-.4){.2}{0}{180}
\qline(3.7,-1)(3.7,-.4) 
\psarc(3.5,.4){.2}{180}{0}
\qline(3.7,.4)(3.7,1)
\psframe[linecolor=black](2.5,.4)(3.5,.8)
\pnode(1.2,-1){a1}
\pnode(1.2,1){a2}
\qline(1.6,-.2)(1.6,-1)
\psarc(2.3,-1){.7}{180}{0}
\qline(1.6,.2)(1.6,1)
\psarc(2.3,1){.7}{0}{180}
\pccurve[angleA=90,angleB=90,ncurv=1](3.7,1)(-.7,0)
\pccurve[angleA=-90,angleB=-90,ncurv=1](-.7,0)(3.7,-1)
\end{pspicture}
\\[1em]
+(\beta_{n+1}+\beta_{m+1})&
\begin{pspicture}[shift=-1.4](-1,-2)(4.5,2.5)
\psframe(-.5,-.2)(.5,.2)
\psframe(1,-.2)(2,.2)
\pccurve[angleA=60,angleB=120,ncurv=1](.1,.2)(1.4,.2)
\pccurve[angleA=-60,angleB=-120,ncurv=1](.1,-.2)(1.4,-.2)
\rput(.75,.8){$m$} \rput(2.25,.8){$n$} 
\qline(3,-1)(3,-.8) \psframe[linecolor=black](2.5,-.8)(3.5,-.4)
\qline(3,-.4)(3,.4)  \qline(3,.8)(3,1)
\psline(3.3,-.4)(3.7,.4)
\qline(3.7,-1)(3.7,-.4) 
\psline(3.3,.4)(3.7,-.4)
\qline(3.7,.4)(3.7,1)
\psframe[linecolor=black](2.5,.4)(3.5,.8)
\pnode(1.2,-1){a1}
\pnode(1.2,1){a2}
\qline(1.6,-.2)(1.6,-1)
\psarc(2.3,-1){.7}{180}{0}
\qline(1.6,.2)(1.6,1)
\psarc(2.3,1){.7}{0}{180}
\pccurve[angleA=90,angleB=90,ncurv=1](3.7,1)(-.7,0)
\pccurve[angleA=-90,angleB=-90,ncurv=1](-.7,0)(3.7,-1)
\end{pspicture}
+\alpha_{n+1}\alpha_{m+1}
\begin{pspicture}[shift=-1.4](-1,-2)(4.5,2)
\psframe(1,-.2)(2,.2)
\rput(.75,.8){$m$} \rput(2.25,.8){$n$} 
\qline(3,-1)(3,-.8) \psframe[linecolor=black](2.5,-.8)(3.5,-.4)
\qline(3,-.4)(3,.4)  \qline(3,.8)(3,1)
 \psarc(3.5,-.4){.2}{0}{180}
\qline(3.7,-1)(3.7,-.4) 
\psarc(3.5,.4){.2}{180}{0}
\qline(3.7,.4)(3.7,1)
\psframe[linecolor=black](2.5,.4)(3.5,.8)
\qline(1.6,-.2)(1.6,-1)
\psarc(2.3,-1){.7}{180}{0}
\qline(1.6,.2)(1.6,1)
\psarc(2.3,1){.7}{0}{180}
\pccurve[angleA=90,angleB=90,ncurv=1](3.7,1)(-.7,1)
\pccurve[angleA=-90,angleB=-90,ncurv=1](-.7,-1)(3.7,-1)
\qline(0,-1)(0,-.8) \psframe[linecolor=black](-.5,-.8)(.5,-.4)
\qline(0,-.4)(0,.4)  \qline(0,.8)(0,1)
 \psarc(-.5,-.4){.2}{0}{180}
\qline(-.7,-1)(-.7,-.4) 
\psarc(-.5,.4){.2}{180}{0}
\qline(-.7,.4)(-.7,1)
\psframe[linecolor=black](-.5,.4)(.5,.8)
\qline(1.4,-.2)(1.4,-1)
\psarc(.7,-1){.7}{180}{0}
\qline(1.4,.2)(1.4,1)
\psarc(.7,1){.7}{0}{180}
\end{pspicture}
\\[3em]
+(\alpha_{m+1}\beta_{n+1}+\beta_{m+1}\alpha_{n+1})&
\begin{pspicture}[shift=-1.4](-1,-2)(4.5,2)
\psframe(1,-.2)(2,.2)
\rput(.75,.8){$m$} \rput(2.25,.8){$n$} 
\qline(3,-1)(3,-.8) \psframe[linecolor=black](2.5,-.8)(3.5,-.4)
\qline(3,-.4)(3,.4)  \qline(3,.8)(3,1)
\qline(3.7,-1)(3.7,-.4) 
\qline(3.7,.4)(3.3,-.4)
\qline(3.7,-.4)(3.3,.4)
\qline(3.7,.4)(3.7,1)
\psframe[linecolor=black](2.5,.4)(3.5,.8)
\qline(1.6,-.2)(1.6,-1)
\psarc(2.3,-1){.7}{180}{0}
\qline(1.6,.2)(1.6,1)
\psarc(2.3,1){.7}{0}{180}
\pccurve[angleA=90,angleB=90,ncurv=1](3.7,1)(-.7,1)
\pccurve[angleA=-90,angleB=-90,ncurv=1](-.7,-1)(3.7,-1)
\qline(0,-1)(0,-.8) \psframe[linecolor=black](-.5,-.8)(.5,-.4)
\qline(0,-.4)(0,.4)  \qline(0,.8)(0,1)
 \psarc(-.5,-.4){.2}{0}{180}
\qline(-.7,-1)(-.7,-.4) 
\psarc(-.5,.4){.2}{180}{0}
\qline(-.7,.4)(-.7,1)
\psframe[linecolor=black](-.5,.4)(.5,.8)
\qline(1.4,-.2)(1.4,-1)
\psarc(.7,-1){.7}{180}{0}
\qline(1.4,.2)(1.4,1)
\psarc(.7,1){.7}{0}{180}
\end{pspicture}
+(\beta_{n+1}\beta_{m+1})
\begin{pspicture}[shift=-1.4](-1,-2)(4.5,2)
\psframe(1,-.2)(2,.2)
\rput(.75,.8){$m$} \rput(2.25,.8){$n$} 
\qline(3,-1)(3,-.8) \psframe[linecolor=black](2.5,-.8)(3.5,-.4)
\qline(3,-.4)(3,.4)  \qline(3,.8)(3,1)
\qline(3.7,-1)(3.7,-.4) 
\qline(3.7,.4)(3.3,-.4)
\qline(3.7,-.4)(3.3,.4)
\qline(3.7,.4)(3.7,1)
\psframe[linecolor=black](2.5,.4)(3.5,.8)
\qline(1.6,-.2)(1.6,-1)
\psarc(2.3,-1){.7}{180}{0}
\qline(1.6,.2)(1.6,1)
\psarc(2.3,1){.7}{0}{180}
\pccurve[angleA=90,angleB=90,ncurv=1](3.7,1)(-.7,1)
\pccurve[angleA=-90,angleB=-90,ncurv=1](-.7,-1)(3.7,-1)
\qline(0,-1)(0,-.8) \psframe[linecolor=black](-.5,-.8)(.5,-.4)
\qline(0,-.4)(0,.4)  \qline(0,.8)(0,1)
\qline(-.7,-1)(-.7,-.4) 
\qline(-.7,.4)(-.3,-.4)
\qline(-.7,-.4)(-.3,.4)
\qline(-.7,.4)(-.7,1)
\psframe[linecolor=black](-.5,.4)(.5,.8)
\qline(1.4,-.2)(1.4,-1)
\psarc(.7,-1){.7}{180}{0}
\qline(1.4,.2)(1.4,1)
\psarc(.7,1){.7}{0}{180}
\end{pspicture}
\end{align*}

\vspace{5 mm}

where the diagrams with sums are collected by symmetry.
One can easily check that the diagrams with $\alpha_{n+1}\alpha_{m+1}$ and $\beta_{n+1}+\beta_{m+1}$ annihilate by the cut path property;  the  diagram with $\alpha_{n+1}+\alpha_{m+1}$ is just $Net(m,n,0)$; the diagram with $\beta_{n+1}+\beta_{m+1}$ is just $\frac{[6][2]}{[3]}Net(m,n,0)$, and the first diagram is just $-\frac{[6][2]}{[3]} Net(m,n,0)$. As for the last diagram with $\beta_{n+1}\beta_{m+1}$, we use the following important trick (which we will continue to use liberally without mention):

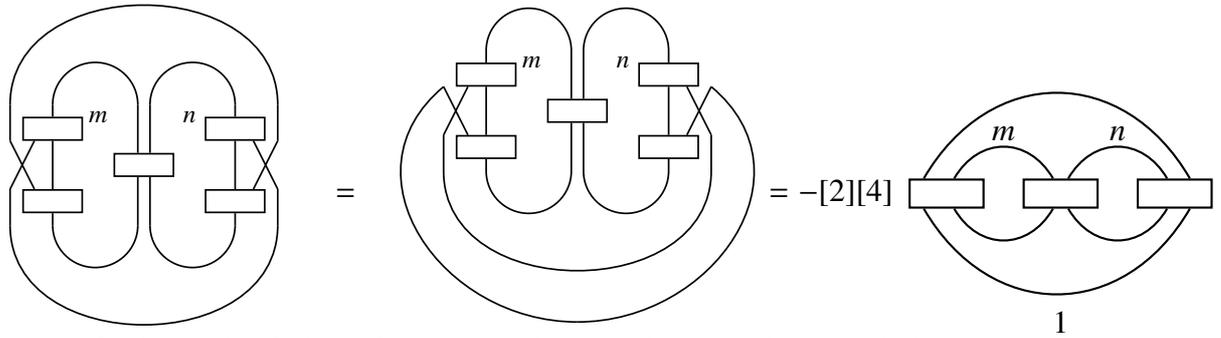
\begin{align*}
\psscalebox{.8}{
\begin{pspicture}[shift=-1.4](-1,-2)(4.5,2)
\psframe(1,-.2)(2,.2)
\rput(.75,.8){$m$} \rput(2.25,.8){$n$} 
\qline(3,-1)(3,-.8) \psframe[linecolor=black](2.5,-.8)(3.5,-.4)
\qline(3,-.4)(3,.4)  \qline(3,.8)(3,1)
\qline(3.7,-1)(3.7,-.4) 
\qline(3.7,.4)(3.3,-.4)
\qline(3.7,-.4)(3.3,.4)
\qline(3.7,.4)(3.7,1)
\psframe[linecolor=black](2.5,.4)(3.5,.8)
\qline(1.6,-.2)(1.6,-1)
\psarc(2.3,-1){.7}{180}{0}
\qline(1.6,.2)(1.6,1)
\psarc(2.3,1){.7}{0}{180}
\pccurve[angleA=90,angleB=90,ncurv=1](3.7,1)(-.7,1)
\pccurve[angleA=-90,angleB=-90,ncurv=1](-.7,-1)(3.7,-1)
\qline(0,-1)(0,-.8) \psframe[linecolor=black](-.5,-.8)(.5,-.4)
\qline(0,-.4)(0,.4)  \qline(0,.8)(0,1)
\qline(-.7,-1)(-.7,-.4) 
\qline(-.7,.4)(-.3,-.4)
\qline(-.7,-.4)(-.3,.4)
\qline(-.7,.4)(-.7,1)
\psframe[linecolor=black](-.5,.4)(.5,.8)
\qline(1.4,-.2)(1.4,-1)
\psarc(.7,-1){.7}{180}{0}
\qline(1.4,.2)(1.4,1)
\psarc(.7,1){.7}{0}{180}
\end{pspicture}
}
&=\psscalebox{.8}{\begin{pspicture}[shift=-2](-2,-3.5)(4.5,2)
\psframe(1,-.2)(2,.2)
\rput(.75,.8){$m$} \rput(2.25,.8){$n$} 
\qline(3,-1)(3,-.8) \psframe[linecolor=black](2.5,-.8)(3.5,-.4)
\qline(3,-.4)(3,.4)  \qline(3,.8)(3,1)
\qline(3.7,.4)(3.3,-.4)
\qline(3.7,-.4)(3.3,.4)
\psframe[linecolor=black](2.5,.4)(3.5,.8)
\qline(1.6,-.2)(1.6,-1)
\psarc(2.3,-1){.7}{180}{0}
\qline(1.6,.2)(1.6,1)
\psarc(2.3,1){.7}{0}{180}
\pccurve[angleA=320,angleB=0,ncurv=1](3.7,.4)(1.5,-3.5)
\pccurve[angleA=180,angleB=220,ncurv=1](1.5,-3.5)(-.7,.4)
\pccurve[angleA=-90,angleB=-90,ncurv=1](-.7,-1)(3.7,-1)
\qline(0,-1)(0,-.8) \psframe[linecolor=black](-.5,-.8)(.5,-.4)
\qline(0,-.4)(0,.4)  \qline(0,.8)(0,1)
\qline(-.7,.4)(-.3,-.4)
\qline(-.7,-.4)(-.3,.4)
\qline(-.7,-.4)(-.7,-1)
\qline(3.7,-.4)(3.7,-1)
\psframe[linecolor=black](-.5,.4)(.5,.8)
\qline(1.4,-.2)(1.4,-1)
\psarc(.7,-1){.7}{180}{0}
\qline(1.4,.2)(1.4,1)
\psarc(.7,1){.7}{0}{180}
\end{pspicture}
}
=-[2][4] \hspace{2mm}\begin{pspicture}[shift=-.9](-.5,-1)(3.5,1)
\psframe(-.5,-.2)(.5,.2)
\psframe(1,-.2)(2,.2)
\psframe(2.5,-.2)(3.5,.2)
\pccurve[angleA=60,angleB=120,ncurv=1](.1,.2)(1.4,.2)
\pccurve[angleA=-60,angleB=-120,ncurv=1](.1,-.2)(1.4,-.2)
\pccurve[angleA=60,angleB=120,ncurv=1](1.6,.2)(2.9,.2)
\pccurve[angleA=-60,angleB=-120,ncurv=1](1.6,-.2)(2.9,-.2)
\pccurve[angleA=60,angleB=120,ncurv=1](-.3,.2)(3.2,.2)
\pccurve[angleA=-60,angleB=-120,ncurv=1](-.3,-.2)(3.2,-.2)
\rput(.75,.8){$m$} \rput(2.25,.8){$n$} \rput(1.5,-1.7){$1$}
\end{pspicture}
\end{align*}
where the final equality follows from expanding the diagram with the relation
\begin{align*}
\begin{pspicture}[shift=-.6](-1.2,-.7)(1.2,.7) \pnode(1.2;30){a1}
\pnode(1.2;150){a2} \pnode(1.2;210){a3} \pnode(1.2;330){a4}
\pnode(.6;90){b1} \pnode(.3;90){b2} \pnode(.3;270){b3}
\nccurve[angleA=315,angleB=180,ncurv=1]{a2}{b3}
\nccurve[angleA=225,angleB=0,ncurv=1]{a1}{b3}
\nccurve[angleA=135,angleB=0,ncurv=1]{a4}{b2}
\nccurve[angleA=45,angleB=180,ncurv=1]{a3}{b2}
\end{pspicture}
 =  -[2]^2\begin{pspicture}[shift=-.6](-.8,-.7)(.8,.7) \pnode(.84;45){a1}
\pnode(.84;135){a2} \pnode(.84;225){a3} \pnode(.84;315){a4}
\ncline{a2}{a4} \ncline{a3}{a1} \end{pspicture} -[2][4]
\begin{pspicture}[shift=-.6](-.7,-.7)(.7,.7) \pnode(.84;45){a1}
\pnode(.84;135){a2} \pnode(.84;225){a3} \pnode(.84;315){a4}
\nccurve[angleA=135,angleB=225,ncurv=1]{a4}{a1}
\nccurve[angleA=45,angleB=-45,ncurv=1]{a3}{a2}
\end{pspicture}
\end{align*}
and noting that the first summand dies by the cut path property.

Putting all of this together, we see that
\begin{align*}Net(m,n,1)=(-\frac{[6]}{[2]}{[3]}+\alpha_{n+1}+\alpha_{m+1}+\frac{[6]}{[2]}{[3]}(\beta_{n+1}+\beta_{m+1})-[4][2]\beta_{n+1}\beta_{m+1})Net(m,n,0)=A_1 net(m,n,0),
\end{align*}
as desired.
\end{proof}

Our next goal is to  determine the value of $Net(m,n,p)$ inductively. Unfortunately, it is too hopeful that this can done directly and for this end we must define a slightly new type of net shape

\vspace{8 mm}

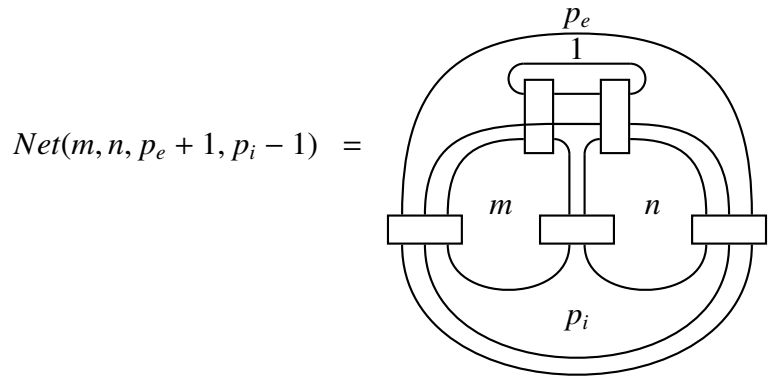
\begin{definition}
\begin{align*}
Net(m,n,p_e+1,p_i-1)\hspace{2mm}=\hspace{2mm}
\begin{pspicture}[shift=-3](-1,-2)(4.5,2)
\psframe(-1,-.2)(0,.2)
\psframe(1,-.2)(2,.2)
\psframe(3,-.2)(4,.2)
\psframe(.8,1)(1.2,2)
\psframe(1.8,1)(2.2,2)
 \psarc(2.2,2){.2}{270}{90}
 \psarc(.8,2){.2}{90}{270}
 \psline(1.2,1.8)(1.8,1.8)
 \psline(.8,2.2)(2.2,2.2)
 \pccurve[angleA=90,angleB=180,ncurv=1](-.2,.2)(.8,1.2)
 \pccurve[angleA=270,angleB=270,ncurv=1](-.2,-.2)(1.4,-.2)
 \pccurve[angleA=0,angleB=90,ncurv=1](2.2,1.2)(3.2,.2)
 \psarc(1.2,1){.2}{0}{90}
 \psline(1.4,1)(1.4,.2)
 \psarc(1.8,1){.2}{90}{180}
 \psline(1.6,1)(1.6,.2)
 \pccurve[angleA=270,angleB=270,ncurv=1](1.6,-.2)(3.2,-.2)
 \rput(1.5,2.4){$1$} \rput(.5,.3){$m$} \rput(2.5,.3){$n$}
\psline(1.2,1.4)(1.8,1.4)
\pccurve[angleA=0,angleB=90,ncurv=1](2.2,1.4)(3.5,.2)
\pccurve[angleA=270,angleB=270,ncurv=1](-.5,-.2)(3.5,-.2)
\pccurve[angleA=180,angleB=90,ncurv=1](1.8,1.4)(-.5,.2)
\rput(1.5,-1.2){$p_i$}
\pccurve[angleA=-90,angleB=-90,ncurv=1](-.8,-.2)(3.8,-.2)
\pccurve[angleA=90,angleB=180,ncurv=1](-.8,.2)(1.5,2.6)
\pccurve[angleA=90,angleB=0,ncurv=1](3.8,.2)(1.5,2.6)
\rput(1.5,2.8){$p_e$}
\end{pspicture}
\end{align*}

where $p_i+p_e=p-1$.
\end{definition}

Equipped with this, we can state and prove the next lemma, where we will care especially about the case $p_i=1$ and $p_e=p-2$.

\begin{lemma} $Net(m,n,p)=A_p Net(m,n,p-1)+B_p Net(m,n,1,p-2).$
\end{lemma}
\begin{proof} The proof method here is very similar, and we begin by isolating the outermost strands into $p-1$ and $1$ to obtain that

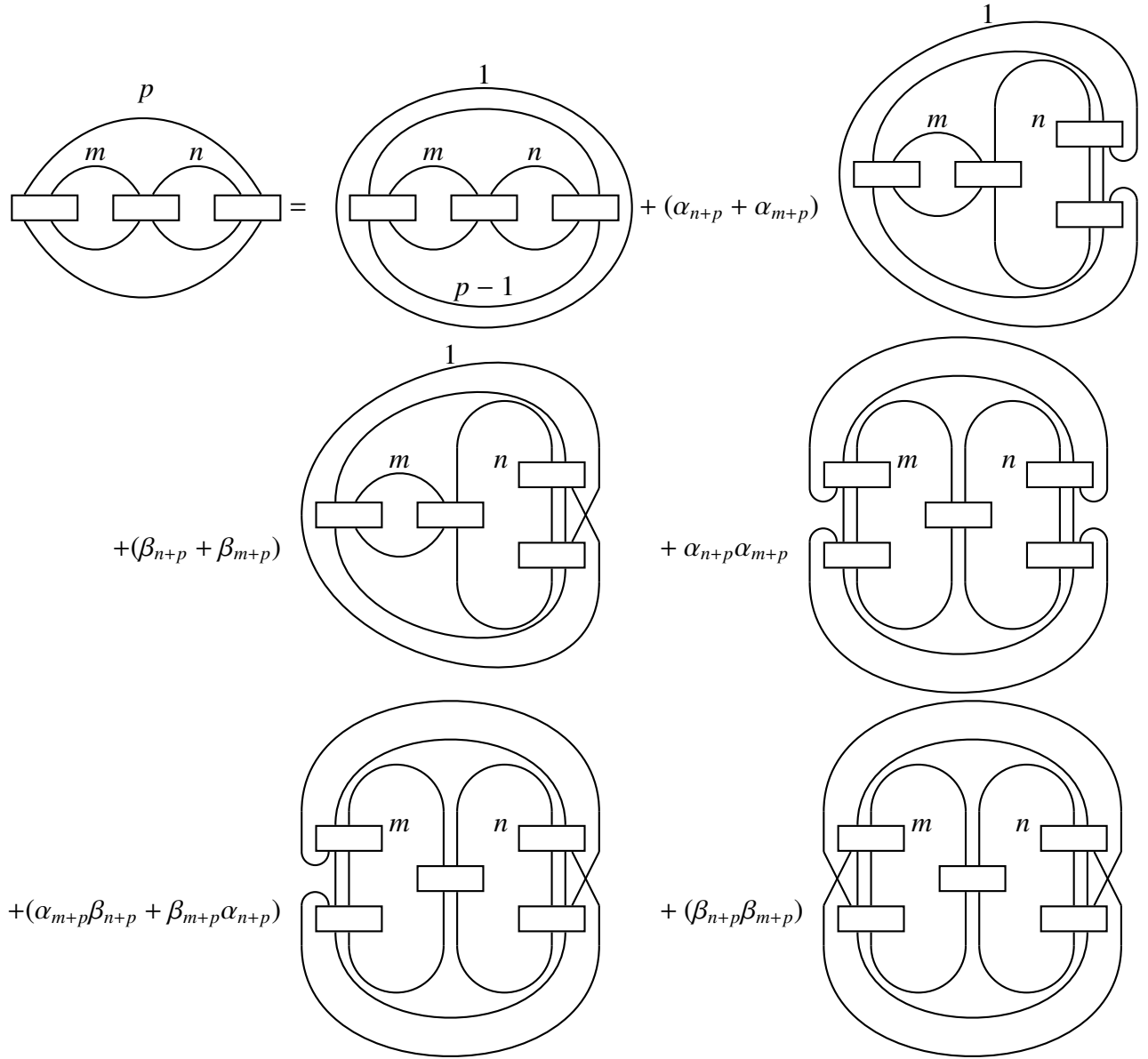
\begin{align*}
\begin{pspicture}[shift=-.9](-.5,-1)(3.5,1)
\psframe(-.5,-.2)(.5,.2)
\psframe(1,-.2)(2,.2)
\psframe(2.5,-.2)(3.5,.2)
\pccurve[angleA=60,angleB=120,ncurv=1](.1,.2)(1.4,.2)
\pccurve[angleA=-60,angleB=-120,ncurv=1](.1,-.2)(1.4,-.2)
\pccurve[angleA=60,angleB=120,ncurv=1](1.6,.2)(2.9,.2)
\pccurve[angleA=-60,angleB=-120,ncurv=1](1.6,-.2)(2.9,-.2)
\pccurve[angleA=60,angleB=120,ncurv=1](-.3,.2)(3.2,.2)
\pccurve[angleA=-60,angleB=-120,ncurv=1](-.3,-.2)(3.2,-.2)
\rput(.75,.8){$m$} \rput(2.25,.8){$n$} \rput(1.5,1.7){$p$}
\end{pspicture}
&=
\begin{pspicture}[shift=-.9](-1,-1)(3.7,1.2)
\psframe(-.5,-.2)(.5,.2)
\psframe(1,-.2)(2,.2)
\psframe(2.5,-.2)(3.5,.2)
\pccurve[angleA=60,angleB=120,ncurv=1](.1,.2)(1.4,.2)
\pccurve[angleA=-60,angleB=-120,ncurv=1](.1,-.2)(1.4,-.2)
\pccurve[angleA=60,angleB=120,ncurv=1](1.6,.2)(2.9,.2)
\pccurve[angleA=-60,angleB=-120,ncurv=1](1.6,-.2)(2.9,-.2)
\rput(.75,.8){$m$} \rput(2.25,.8){$n$} 
\pnode(1.5,0){a1}
\psellipse(1.5,0)(2.2,1.8)
\pccurve[angleA=90,angleB=90,ncurv=1](-.2,.2)(3.2,.2)
\pccurve[angleA=270,angleB=270,ncurv=1](-.2,-.2)(3.2,-.2)
\rput(1.5,2){1} \rput(1.5,-1.2){$p-1$}
\end{pspicture}
+(\alpha_{n+p}+\alpha_{m+p})
\begin{pspicture}[shift=-1.4](-1,-2)(4.5,2)
\psframe(-.5,-.2)(.5,.2)
\psframe(1,-.2)(2,.2)
\pccurve[angleA=60,angleB=120,ncurv=1](.1,.2)(1.4,.2)
\pccurve[angleA=-60,angleB=-120,ncurv=1](.1,-.2)(1.4,-.2)
\rput(.75,.8){$m$} \rput(2.25,.8){$n$} 
\qline(3,-1)(3,-.8) \psframe[linecolor=black](2.5,-.8)(3.5,-.4)
\qline(3,-.4)(3,.4)  \qline(3,.8)(3,1)
 \psarc(3.5,-.4){.2}{0}{180}
\qline(3.7,-1)(3.7,-.4) 
\psarc(3.5,.4){.2}{180}{0}
\qline(3.7,.4)(3.7,1)
\psframe[linecolor=black](2.5,.4)(3.5,.8)
\pnode(1.2,-1){a1}
\pnode(1.2,1){a2}
\qline(1.6,-.2)(1.6,-1)
\psarc(2.3,-1){.7}{180}{0}
\qline(1.6,.2)(1.6,1)
\psarc(2.3,1){.7}{0}{180}
\pccurve[angleA=90,angleB=90,ncurv=1](3.7,1)(-.7,0)
\pccurve[angleA=-90,angleB=-90,ncurv=1](-.7,0)(3.7,-1)
\pccurve[angleA=90,angleB=90,ncurv=1](-.2,.2)(3.2,.8)
\pccurve[angleA=270,angleB=270,ncurv=1](-.2,-.2)(3.2,-.8)
\psline(3.2,.4)(3.2,-.4)
\rput(1.5,2.4){1} 
\end{pspicture}
\\[1em]
+(\beta_{n+p}+\beta_{m+p})&
\begin{pspicture}[shift=-1.4](-1,-2)(4.5,2.5)
\psframe(-.5,-.2)(.5,.2)
\psframe(1,-.2)(2,.2)
\pccurve[angleA=60,angleB=120,ncurv=1](.1,.2)(1.4,.2)
\pccurve[angleA=-60,angleB=-120,ncurv=1](.1,-.2)(1.4,-.2)
\rput(.75,.8){$m$} \rput(2.25,.8){$n$} 
\qline(3,-1)(3,-.8) \psframe[linecolor=black](2.5,-.8)(3.5,-.4)
\qline(3,-.4)(3,.4)  \qline(3,.8)(3,1)
\psline(3.3,-.4)(3.7,.4)
\qline(3.7,-1)(3.7,-.4) 
\psline(3.3,.4)(3.7,-.4)
\qline(3.7,.4)(3.7,1)
\psframe[linecolor=black](2.5,.4)(3.5,.8)
\pnode(1.2,-1){a1}
\pnode(1.2,1){a2}
\qline(1.6,-.2)(1.6,-1)
\psarc(2.3,-1){.7}{180}{0}
\qline(1.6,.2)(1.6,1)
\psarc(2.3,1){.7}{0}{180}
\pccurve[angleA=90,angleB=90,ncurv=1](3.7,1)(-.7,0)
\pccurve[angleA=-90,angleB=-90,ncurv=1](-.7,0)(3.7,-1)
\pccurve[angleA=90,angleB=90,ncurv=1](-.2,.2)(3.2,.8)
\pccurve[angleA=270,angleB=270,ncurv=1](-.2,-.2)(3.2,-.8)
\psline(3.2,.4)(3.2,-.4)
\rput(1.5,2.4){1} 
\end{pspicture}
+\alpha_{n+p}\alpha_{m+p}
\begin{pspicture}[shift=-1.4](-1,-2)(4.5,2)
\psframe(1,-.2)(2,.2)
\rput(.75,.8){$m$} \rput(2.25,.8){$n$} 
\qline(3,-1)(3,-.8) \psframe[linecolor=black](2.5,-.8)(3.5,-.4)
\qline(3,-.4)(3,.4)  \qline(3,.8)(3,1)
 \psarc(3.5,-.4){.2}{0}{180}
\qline(3.7,-1)(3.7,-.4) 
\psarc(3.5,.4){.2}{180}{0}
\qline(3.7,.4)(3.7,1)
\psframe[linecolor=black](2.5,.4)(3.5,.8)
\qline(1.6,-.2)(1.6,-1)
\psarc(2.3,-1){.7}{180}{0}
\qline(1.6,.2)(1.6,1)
\psarc(2.3,1){.7}{0}{180}
\pccurve[angleA=90,angleB=90,ncurv=1](3.7,1)(-.7,1)
\pccurve[angleA=-90,angleB=-90,ncurv=1](-.7,-1)(3.7,-1)
\qline(0,-1)(0,-.8) \psframe[linecolor=black](-.5,-.8)(.5,-.4)
\qline(0,-.4)(0,.4)  \qline(0,.8)(0,1)
 \psarc(-.5,-.4){.2}{0}{180}
\qline(-.7,-1)(-.7,-.4) 
\psarc(-.5,.4){.2}{180}{0}
\qline(-.7,.4)(-.7,1)
\psframe[linecolor=black](-.5,.4)(.5,.8)
\qline(1.4,-.2)(1.4,-1)
\psarc(.7,-1){.7}{180}{0}
\qline(1.4,.2)(1.4,1)
\psarc(.7,1){.7}{0}{180}
\pccurve[angleA=90,angleB=90,ncurv=1](-.2,.8)(3.2,.8)
\pccurve[angleA=270,angleB=270,ncurv=1](-.2,-.8)(3.2,-.8)
\psline(3.2,.4)(3.2,-.4)
\psline(-.2,.4)(-.2,-.4)
\end{pspicture}
\\[3em]
+(\alpha_{m+p}\beta_{n+p}+\beta_{m+p}\alpha_{n+p})&
\begin{pspicture}[shift=-1.4](-1,-2)(4.5,2)
\psframe(1,-.2)(2,.2)
\rput(.75,.8){$m$} \rput(2.25,.8){$n$} 
\qline(3,-1)(3,-.8) \psframe[linecolor=black](2.5,-.8)(3.5,-.4)
\qline(3,-.4)(3,.4)  \qline(3,.8)(3,1)
\qline(3.7,-1)(3.7,-.4) 
\qline(3.7,.4)(3.3,-.4)
\qline(3.7,-.4)(3.3,.4)
\qline(3.7,.4)(3.7,1)
\psframe[linecolor=black](2.5,.4)(3.5,.8)
\qline(1.6,-.2)(1.6,-1)
\psarc(2.3,-1){.7}{180}{0}
\qline(1.6,.2)(1.6,1)
\psarc(2.3,1){.7}{0}{180}
\pccurve[angleA=90,angleB=90,ncurv=1](3.7,1)(-.7,1)
\pccurve[angleA=-90,angleB=-90,ncurv=1](-.7,-1)(3.7,-1)
\qline(0,-1)(0,-.8) \psframe[linecolor=black](-.5,-.8)(.5,-.4)
\qline(0,-.4)(0,.4)  \qline(0,.8)(0,1)
 \psarc(-.5,-.4){.2}{0}{180}
\qline(-.7,-1)(-.7,-.4) 
\psarc(-.5,.4){.2}{180}{0}
\qline(-.7,.4)(-.7,1)
\psframe[linecolor=black](-.5,.4)(.5,.8)
\qline(1.4,-.2)(1.4,-1)
\psarc(.7,-1){.7}{180}{0}
\qline(1.4,.2)(1.4,1)
\psarc(.7,1){.7}{0}{180}
\pccurve[angleA=90,angleB=90,ncurv=1](-.2,.8)(3.2,.8)
\pccurve[angleA=270,angleB=270,ncurv=1](-.2,-.8)(3.2,-.8)
\psline(3.2,.4)(3.2,-.4)
\psline(-.2,.4)(-.2,-.4)
\end{pspicture}
+(\beta_{n+p}\beta_{m+p})
\begin{pspicture}[shift=-1.4](-1,-2)(4.5,2)
\psframe(1,-.2)(2,.2)
\rput(.75,.8){$m$} \rput(2.25,.8){$n$} 
\qline(3,-1)(3,-.8) \psframe[linecolor=black](2.5,-.8)(3.5,-.4)
\qline(3,-.4)(3,.4)  \qline(3,.8)(3,1)
\qline(3.7,-1)(3.7,-.4) 
\qline(3.7,.4)(3.3,-.4)
\qline(3.7,-.4)(3.3,.4)
\qline(3.7,.4)(3.7,1)
\psframe[linecolor=black](2.5,.4)(3.5,.8)
\qline(1.6,-.2)(1.6,-1)
\psarc(2.3,-1){.7}{180}{0}
\qline(1.6,.2)(1.6,1)
\psarc(2.3,1){.7}{0}{180}
\pccurve[angleA=90,angleB=90,ncurv=1](3.7,1)(-.7,1)
\pccurve[angleA=-90,angleB=-90,ncurv=1](-.7,-1)(3.7,-1)
\qline(0,-1)(0,-.8) \psframe[linecolor=black](-.5,-.8)(.5,-.4)
\qline(0,-.4)(0,.4)  \qline(0,.8)(0,1)
\qline(-.7,-1)(-.7,-.4) 
\qline(-.7,.4)(-.3,-.4)
\qline(-.7,-.4)(-.3,.4)
\qline(-.7,.4)(-.7,1)
\psframe[linecolor=black](-.5,.4)(.5,.8)
\qline(1.4,-.2)(1.4,-1)
\psarc(.7,-1){.7}{180}{0}
\qline(1.4,.2)(1.4,1)
\psarc(.7,1){.7}{0}{180}
\pccurve[angleA=90,angleB=90,ncurv=1](-.2,.8)(3.2,.8)
\pccurve[angleA=270,angleB=270,ncurv=1](-.2,-.8)(3.2,-.8)
\psline(3.2,.4)(3.2,-.4)
\psline(-.2,.4)(-.2,-.4)
\end{pspicture}
\end{align*}

\vspace{4 mm}

We handle the first three summands precisely as before and notice that the last one can also be handled with the ``double cross trick.'' Collecting terms, we see that we obtain precisely $A_p Net(m,n,p-1)$. The fourth summand is precisely $Net(m,n,1,p-2)$ up to isotopy, giving us the term $B_p Net(m,n,1,p-2)$ as claimed. Finally, the penultimate summand dies by the cut path property, proving the claim.
\end{proof}
Our idea will now be to calculate $Net(m,n,p-1,0)$ and to reduce our calculation of $Net(m,n,1,p-2)$ to this case by recursively expressing $Net(m,n,p_e,p_i)$ in terms of $Net(m,n,p_e+1,p_i-1)$. To this end, we prove the following two lemmas

\begin{lemma}
$Net(m,n,p-1,0)=A_1 Net(m,n,p-1)$
\end{lemma}

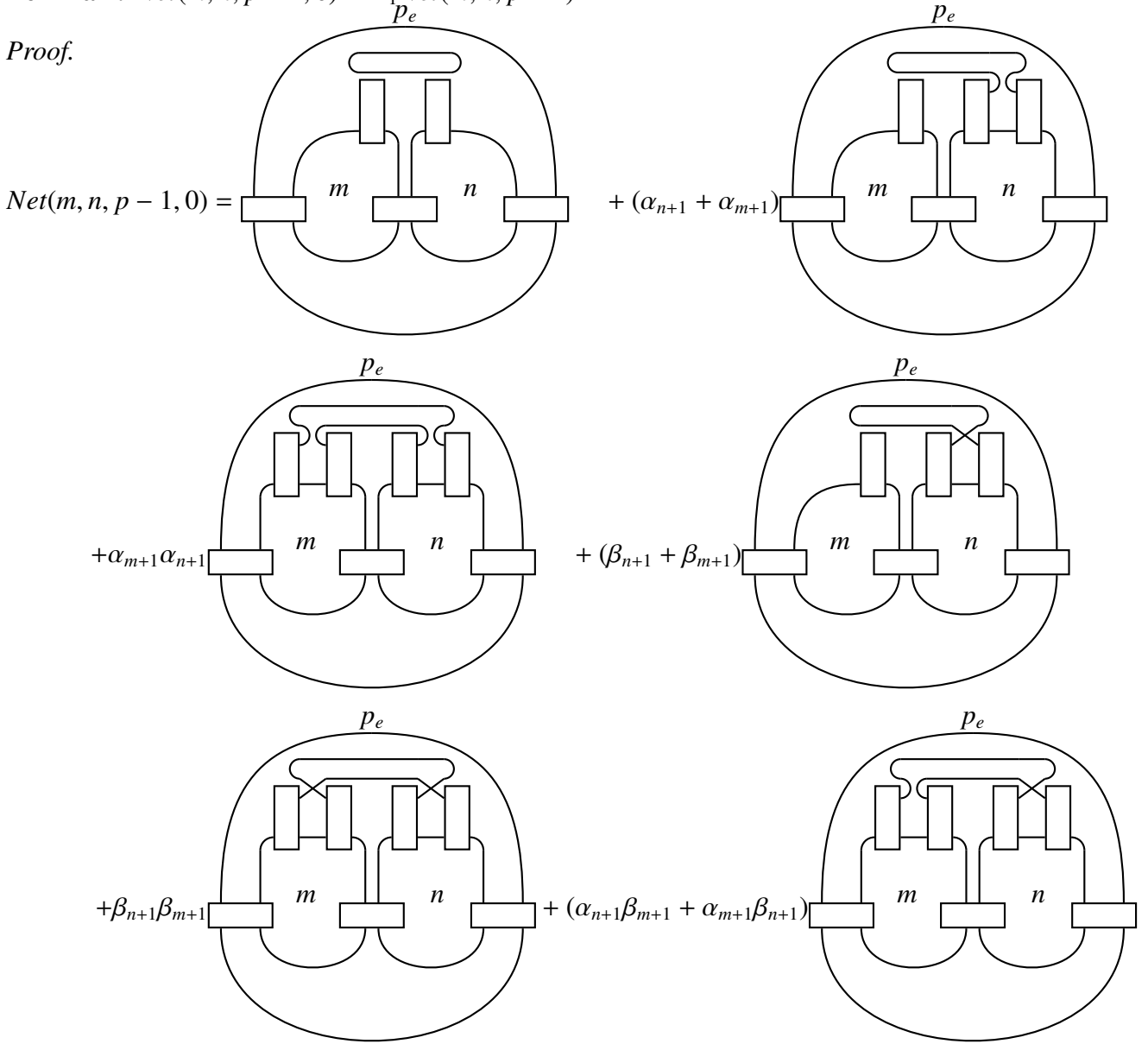
\begin{proof}

\begin{align*}
Net(m,n,p-1,0)
&=
\begin{pspicture}[shift=-2](-1,-2)(4.5,2)
\psframe(-1,-.2)(0,.2)
\psframe(1,-.2)(2,.2)
\psframe(3,-.2)(4,.2)
\psframe(.8,1)(1.2,2)
\psframe(1.8,1)(2.2,2)
 \psarc(2.2,2.25){.15}{270}{90}
 \psarc(.8,2.25){.15}{90}{270}
 \psline(.8,2.1)(2.2,2.1)
 \psline(.8,2.4)(2.2,2.4)
 \pccurve[angleA=90,angleB=180,ncurv=1](-.2,.2)(.8,1.2)
 \pccurve[angleA=270,angleB=270,ncurv=1](-.2,-.2)(1.4,-.2)
 \pccurve[angleA=0,angleB=90,ncurv=1](2.2,1.2)(3.2,.2)
 \psarc(1.2,1){.2}{0}{90}
 \psline(1.4,1)(1.4,.2)
 \psarc(1.8,1){.2}{90}{180}
 \psline(1.6,1)(1.6,.2)
 \pccurve[angleA=270,angleB=270,ncurv=1](1.6,-.2)(3.2,-.2)
 \rput(.5,.3){$m$} \rput(2.5,.3){$n$}
\pccurve[angleA=-90,angleB=-90,ncurv=1](-.8,-.2)(3.8,-.2)
\pccurve[angleA=90,angleB=180,ncurv=1](-.8,.2)(1.5,2.8)
\pccurve[angleA=90,angleB=0,ncurv=1](3.8,.2)(1.5,2.8)
\rput(1.5,3){$p_e$}
\end{pspicture}
+
(\alpha_{n+1}+\alpha_{m+1})
\begin{pspicture}[shift=-2](-1,-2)(4.5,2)
\psframe(-1,-.2)(0,.2)
\psframe(1,-.2)(2,.2)
\psframe(3,-.2)(4,.2)
\psframe(.8,1)(1.2,2)
\psframe(1.8,1)(2.2,2)
\psframe(2.6,1)(3,2)
 \psarc(.8,2.25){.15}{90}{270}
 \psline(.8,2.1)(2.2,2.1)
 \psline(.8,2.4)(2.2,2.4)
 \psline(2.2,1.2)(2.6,1.2)
 \pccurve[angleA=90,angleB=180,ncurv=1](-.2,.2)(.8,1.2)
 \pccurve[angleA=270,angleB=270,ncurv=1](-.2,-.2)(1.4,-.2)
 \psarc(1.2,1){.2}{0}{90}
 \psline(1.4,1)(1.4,.2)
 \psarc(1.8,1){.2}{90}{180}
 \psline(1.6,1)(1.6,.2)
 \pccurve[angleA=270,angleB=270,ncurv=1](1.6,-.2)(3.2,-.2)
 \rput(.5,.3){$m$} \rput(2.5,.3){$n$}
\pccurve[angleA=-90,angleB=-90,ncurv=1](-.8,-.2)(3.8,-.2)
\pccurve[angleA=90,angleB=180,ncurv=1](-.8,.2)(1.5,2.8)
\pccurve[angleA=90,angleB=0,ncurv=1](3.8,.2)(1.5,2.8)
\rput(1.5,3){$p_e$}
\psarc(3,1){.2}{0}{90}
\psline(3.2,1)(3.2,.2)
\psarc(2.2,1.95){.15}{270}{90}
\psarc(2.6,1.95){.15}{90}{270}
\psarc(2.6,2.25){.15}{270}{90}
\psline(2.6,2.4)(2.2,2.4)
\end{pspicture}
\\[3em]
+
\alpha_{m+1}\alpha_{n+1}&\begin{pspicture}[shift=-2](-1,-2)(4.5,2)
\psframe(-1,-.2)(0,.2)
\psframe(1,-.2)(2,.2)
\psframe(3,-.2)(4,.2)
\psframe(.8,1)(1.2,2)
\psframe(1.8,1)(2.2,2)
\psframe(0,1)(.4,2)
\psframe(2.6,1)(3,2)
 \psline(.8,2.1)(2.2,2.1)
 \pccurve[angleA=270,angleB=270,ncurv=1](-.2,-.2)(1.4,-.2)
 \psarc(1.2,1){.2}{0}{90}
 \psline(1.4,1)(1.4,.2)
 \psarc(1.8,1){.2}{90}{180}
 \psline(1.6,1)(1.6,.2)
 \pccurve[angleA=270,angleB=270,ncurv=1](1.6,-.2)(3.2,-.2)
 \rput(.5,.3){$m$} \rput(2.5,.3){$n$}
\pccurve[angleA=-90,angleB=-90,ncurv=1](-.8,-.2)(3.8,-.2)
\pccurve[angleA=90,angleB=180,ncurv=1](-.8,.2)(1.5,2.8)
\pccurve[angleA=90,angleB=0,ncurv=1](3.8,.2)(1.5,2.8)
\rput(1.5,3){$p_e$}
\psline(.4,1.2)(.8,1.2)
\psarc(0,1){.2}{90}{180}
\psline(-.2,1)(-.2,.2)
\psarc(.8,1.95){.15}{90}{270}
\psarc(.4,1.95){.15}{270}{90}
\psarc(.4,2.25){.15}{90}{270}
\psline(.4,2.4)(2.2,2.4)
\psarc(3,1){.2}{0}{90}
\psline(2.2,1.2)(2.6,1.2)
\psline(3.2,1)(3.2,.2)
\psarc(2.2,1.95){.15}{270}{90}
\psarc(2.6,1.95){.15}{90}{270}
\psarc(2.6,2.25){.15}{270}{90}
\psline(2.6,2.4)(2.2,2.4)
\end{pspicture}
+
(\beta_{n+1}+\beta_{m+1})
\begin{pspicture}[shift=-2](-1,-2)(4.5,2)
\psframe(-1,-.2)(0,.2)
\psframe(1,-.2)(2,.2)
\psframe(3,-.2)(4,.2)
\psframe(.8,1)(1.2,2)
\psframe(1.8,1)(2.2,2)
\psframe(2.6,1)(3,2)
 \psarc(.8,2.25){.15}{90}{270}
 \psline(.8,2.1)(2.2,2.1)
 \psline(.8,2.4)(2.2,2.4)
 \psline(2.2,1.2)(2.6,1.2)
 \pccurve[angleA=90,angleB=180,ncurv=1](-.2,.2)(.8,1.2)
 \pccurve[angleA=270,angleB=270,ncurv=1](-.2,-.2)(1.4,-.2)
 \psarc(1.2,1){.2}{0}{90}
 \psline(1.4,1)(1.4,.2)
 \psarc(1.8,1){.2}{90}{180}
 \psline(1.6,1)(1.6,.2)
 \pccurve[angleA=270,angleB=270,ncurv=1](1.6,-.2)(3.2,-.2)
 \rput(.5,.3){$m$} \rput(2.5,.3){$n$}
\pccurve[angleA=-90,angleB=-90,ncurv=1](-.8,-.2)(3.8,-.2)
\pccurve[angleA=90,angleB=180,ncurv=1](-.8,.2)(1.5,2.8)
\pccurve[angleA=90,angleB=0,ncurv=1](3.8,.2)(1.5,2.8)
\rput(1.5,3){$p_e$}
\psarc(3,1){.2}{0}{90}
\psline(3.2,1)(3.2,.2)
\psarc(2.6,2.25){.15}{270}{90}
\psline(2.6,2.4)(2.2,2.4)
\psline(2.2,1.8)(2.6,2.1)
\psline(2.6,1.8)(2.2,2.1)
\end{pspicture}
\\[3em]
+
\beta_{n+1}\beta_{m+1}&\begin{pspicture}[shift=-2](-1,-2)(4.5,2)
\psframe(-1,-.2)(0,.2)
\psframe(1,-.2)(2,.2)
\psframe(3,-.2)(4,.2)
\psframe(.8,1)(1.2,2)
\psframe(1.8,1)(2.2,2)
\psframe(0,1)(.4,2)
\psframe(2.6,1)(3,2)
 \psline(.8,2.1)(2.2,2.1)
 \pccurve[angleA=270,angleB=270,ncurv=1](-.2,-.2)(1.4,-.2)
 \psarc(1.2,1){.2}{0}{90}
 \psline(1.4,1)(1.4,.2)
 \psarc(1.8,1){.2}{90}{180}
 \psline(1.6,1)(1.6,.2)
 \pccurve[angleA=270,angleB=270,ncurv=1](1.6,-.2)(3.2,-.2)
 \rput(.5,.3){$m$} \rput(2.5,.3){$n$}
\pccurve[angleA=-90,angleB=-90,ncurv=1](-.8,-.2)(3.8,-.2)
\pccurve[angleA=90,angleB=180,ncurv=1](-.8,.2)(1.5,2.8)
\pccurve[angleA=90,angleB=0,ncurv=1](3.8,.2)(1.5,2.8)
\rput(1.5,3){$p_e$}
\psline(.4,1.2)(.8,1.2)
\psarc(0,1){.2}{90}{180}
\psline(-.2,1)(-.2,.2)
\psline(.4,2.1)(.8,1.8)
\psline(.8,2.1)(.4,1.8)
\psarc(.4,2.25){.15}{90}{270}
\psline(.4,2.4)(2.2,2.4)
\psarc(3,1){.2}{0}{90}
\psline(2.2,1.2)(2.6,1.2)
\psline(3.2,1)(3.2,.2)
\psline(2.2,1.8)(2.6,2.1)
\psline(2.6,1.8)(2.2,2.1)
\psarc(2.6,2.25){.15}{270}{90}
\psline(2.6,2.4)(2.2,2.4)
\end{pspicture}
\hspace{-5 mm}
+(\alpha_{n+1}\beta_{m+1}+\alpha_{m+1}\beta_{n+1})\begin{pspicture}[shift=-2](-1,-2)(4.5,2)
\psframe(-1,-.2)(0,.2)
\psframe(1,-.2)(2,.2)
\psframe(3,-.2)(4,.2)
\psframe(.8,1)(1.2,2)
\psframe(1.8,1)(2.2,2)
\psframe(0,1)(.4,2)
\psframe(2.6,1)(3,2)
 \psline(.8,2.1)(2.2,2.1)
 \pccurve[angleA=270,angleB=270,ncurv=1](-.2,-.2)(1.4,-.2)
 \psarc(1.2,1){.2}{0}{90}
 \psline(1.4,1)(1.4,.2)
 \psarc(1.8,1){.2}{90}{180}
 \psline(1.6,1)(1.6,.2)
 \pccurve[angleA=270,angleB=270,ncurv=1](1.6,-.2)(3.2,-.2)
 \rput(.5,.3){$m$} \rput(2.5,.3){$n$}
\pccurve[angleA=-90,angleB=-90,ncurv=1](-.8,-.2)(3.8,-.2)
\pccurve[angleA=90,angleB=180,ncurv=1](-.8,.2)(1.5,2.8)
\pccurve[angleA=90,angleB=0,ncurv=1](3.8,.2)(1.5,2.8)
\rput(1.5,3){$p_e$}
\psline(.4,1.2)(.8,1.2)
\psarc(0,1){.2}{90}{180}
\psline(-.2,1)(-.2,.2)
\psarc(.8,1.95){.15}{90}{270}
\psarc(.4,1.95){.15}{270}{90}
\psarc(.4,2.25){.15}{90}{270}
\psline(.4,2.4)(2.2,2.4)
\psarc(3,1){.2}{0}{90}
\psline(2.2,1.2)(2.6,1.2)
\psline(3.2,1)(3.2,.2)
\psline(2.2,1.8)(2.6,2.1)
\psline(2.2,2.1)(2.6,1.8)
\psarc(2.6,2.25){.15}{270}{90}
\psline(2.6,2.4)(2.2,2.4)
\end{pspicture}
\end{align*}
which all reduce exactly as claimed by the annihilation property and calculations similar to those in previous lemmas.

\end{proof}

We now arrive at the final lemma needed for our recursive evaluation:
\begin{lemma} $Net(m,n,p_e,p_i)=A_{p_i+1}Net(m,n,p-1)+B_{p_i+1}Net(m,n,p_{e+1},p_{i+1})$
\end{lemma}
\begin{proof}
\begin{flalign*}
 Net(m,n,p_e,p_i)
&=
\begin{pspicture}[shift=-2](-1,-2)(4.5,2)
\psframe(-1,-.2)(0,.2)
\psframe(1,-.2)(2,.2)
\psframe(3,-.2)(4,.2)
\psframe(.8,1)(1.2,2)
\psframe(1.8,1)(2.2,2)
 \psarc(2.2,2.25){.15}{270}{90}
 \psarc(.8,2.25){.15}{90}{270}
 \psline(.8,2.1)(2.2,2.1)
 \psline(.8,2.4)(2.2,2.4)
 \pccurve[angleA=90,angleB=180,ncurv=1](-.2,.2)(.8,1.2)
 \pccurve[angleA=270,angleB=270,ncurv=1](-.2,-.2)(1.4,-.2)
 \pccurve[angleA=0,angleB=90,ncurv=1](2.2,1.2)(3.2,.2)
 \psarc(1.2,1){.2}{0}{90}
 \psline(1.4,1)(1.4,.2)
 \psarc(1.8,1){.2}{90}{180}
 \psline(1.6,1)(1.6,.2)
 \pccurve[angleA=270,angleB=270,ncurv=1](1.6,-.2)(3.2,-.2)
 \rput(.5,.3){$m$} \rput(2.5,.3){$n$}
\pccurve[angleA=-90,angleB=-90,ncurv=1](-.8,-.2)(3.8,-.2)
\pccurve[angleA=90,angleB=180,ncurv=1](-.8,.2)(1.5,2.8)
\pccurve[angleA=90,angleB=0,ncurv=1](3.8,.2)(1.5,2.8)
\rput(1.5,3){$p_e$}
\pccurve[angleA=-90,angleB=-90,ncurv=1](3.5,-.2)(-.5,-.2)
\pccurve[angleA=90,angleB=0,ncurv=1](3.5,.2)(2.2,1.5)
\psline(1.2,1.5)(1.8,1.5)
\pccurve[angleA=90,angleB=180,ncurv=1](-.5,.2)(.8,1.5)
\rput(1.5,-1.5){$p_i$}
\end{pspicture}
+
(\alpha_{n+1}+\alpha_{m+1}) \hspace{2 mm}
\begin{pspicture}[shift=-2](-1,-2)(4.5,2)
\psframe(-1,-.2)(0,.2)
\psframe(1,-.2)(2,.2)
\psframe(3,-.2)(4,.2)
\psframe(.8,1)(1.2,2)
\psframe(1.8,1)(2.2,2)
\psframe(2.6,1)(3,2)
 \psarc(.8,2.25){.15}{90}{270}
 \psline(.8,2.1)(2.2,2.1)
 \psline(.8,2.4)(2.2,2.4)
 \psline(2.2,1.2)(2.6,1.2)
 \pccurve[angleA=90,angleB=180,ncurv=1](-.2,.2)(.8,1.2)
 \pccurve[angleA=270,angleB=270,ncurv=1](-.2,-.2)(1.4,-.2)
 \psarc(1.2,1){.2}{0}{90}
 \psline(1.4,1)(1.4,.2)
 \psarc(1.8,1){.2}{90}{180}
 \psline(1.6,1)(1.6,.2)
 \pccurve[angleA=270,angleB=270,ncurv=1](1.6,-.2)(3.2,-.2)
 \rput(.5,.3){$m$} \rput(2.5,.3){$n$}
\pccurve[angleA=-90,angleB=-90,ncurv=1](-.8,-.2)(3.8,-.2)
\pccurve[angleA=90,angleB=180,ncurv=1](-.8,.2)(1.5,2.8)
\pccurve[angleA=90,angleB=0,ncurv=1](3.8,.2)(1.5,2.8)
\rput(1.5,3){$p_e$}
\psarc(3,1){.2}{0}{90}
\psline(3.2,1)(3.2,.2)
\psarc(2.2,1.95){.15}{270}{90}
\psarc(2.6,1.95){.15}{90}{270}
\psarc(2.6,2.25){.15}{270}{90}
\psline(2.6,2.4)(2.2,2.4)
\pccurve[angleA=-90,angleB=-90,ncurv=1](3.5,-.2)(-.5,-.2)
\pccurve[angleA=90,angleB=0,ncurv=1](3.5,.2)(3,1.5)
\psline(1.2,1.5)(1.8,1.5)
\pccurve[angleA=90,angleB=180,ncurv=1](-.5,.2)(.8,1.5)
\psline(2.6,1.5)(2.2,1.5)
\rput(1.5,-1.5){$p_i$}
\end{pspicture}
\\[3em]
+
\alpha_{m+1}\alpha_{n+1}&\begin{pspicture}[shift=-2](-1,-2)(4.5,2)
\psframe(-1,-.2)(0,.2)
\psframe(1,-.2)(2,.2)
\psframe(3,-.2)(4,.2)
\psframe(.8,1)(1.2,2)
\psframe(1.8,1)(2.2,2)
\psframe(0,1)(.4,2)
\psframe(2.6,1)(3,2)
 \psline(.8,2.1)(2.2,2.1)
 \pccurve[angleA=270,angleB=270,ncurv=1](-.2,-.2)(1.4,-.2)
 \psarc(1.2,1){.2}{0}{90}
 \psline(1.4,1)(1.4,.2)
 \psarc(1.8,1){.2}{90}{180}
 \psline(1.6,1)(1.6,.2)
 \pccurve[angleA=270,angleB=270,ncurv=1](1.6,-.2)(3.2,-.2)
 \rput(.5,.3){$m$} \rput(2.5,.3){$n$}
\pccurve[angleA=-90,angleB=-90,ncurv=1](-.8,-.2)(3.8,-.2)
\pccurve[angleA=90,angleB=180,ncurv=1](-.8,.2)(1.5,2.8)
\pccurve[angleA=90,angleB=0,ncurv=1](3.8,.2)(1.5,2.8)
\rput(1.5,3){$p_e$}
\psline(.4,1.2)(.8,1.2)
\psarc(0,1){.2}{90}{180}
\psline(-.2,1)(-.2,.2)
\psarc(.8,1.95){.15}{90}{270}
\psarc(.4,1.95){.15}{270}{90}
\psarc(.4,2.25){.15}{90}{270}
\psline(.4,2.4)(2.2,2.4)
\psarc(3,1){.2}{0}{90}
\psline(2.2,1.2)(2.6,1.2)
\psline(3.2,1)(3.2,.2)
\psarc(2.2,1.95){.15}{270}{90}
\psarc(2.6,1.95){.15}{90}{270}
\psarc(2.6,2.25){.15}{270}{90}
\psline(2.6,2.4)(2.2,2.4)
\psarc(3,1){.2}{0}{90}
\psline(3.2,1)(3.2,.2)
\psarc(2.2,1.95){.15}{270}{90}
\psarc(2.6,1.95){.15}{90}{270}
\psarc(2.6,2.25){.15}{270}{90}
\psline(2.6,2.4)(2.2,2.4)
\pccurve[angleA=-90,angleB=-90,ncurv=1](3.5,-.2)(-.5,-.2)
\pccurve[angleA=90,angleB=0,ncurv=1](3.5,.2)(3,1.5)
\psline(1.2,1.5)(1.8,1.5)
\pccurve[angleA=90,angleB=180,ncurv=1](-.5,.2)(0,1.5)
\psline(2.6,1.5)(2.2,1.5)
\psline(.4,1.5)(.8,1.5)
\rput(1.5,-1.5){$p_i$}
\end{pspicture}
+
(\beta_{n+1}+\beta_{m+1}) \hspace{7 mm}
\begin{pspicture}[shift=-2](-1,-2)(4.5,2)
\psframe(-1,-.2)(0,.2)
\psframe(1,-.2)(2,.2)
\psframe(3,-.2)(4,.2)
\psframe(.8,1)(1.2,2)
\psframe(1.8,1)(2.2,2)
\psframe(2.6,1)(3,2)
 \psarc(.8,2.25){.15}{90}{270}
 \psline(.8,2.1)(2.2,2.1)
 \psline(.8,2.4)(2.2,2.4)
 \psline(2.2,1.2)(2.6,1.2)
 \pccurve[angleA=90,angleB=180,ncurv=1](-.2,.2)(.8,1.2)
 \pccurve[angleA=270,angleB=270,ncurv=1](-.2,-.2)(1.4,-.2)
 \psarc(1.2,1){.2}{0}{90}
 \psline(1.4,1)(1.4,.2)
 \psarc(1.8,1){.2}{90}{180}
 \psline(1.6,1)(1.6,.2)
 \pccurve[angleA=270,angleB=270,ncurv=1](1.6,-.2)(3.2,-.2)
 \rput(.5,.3){$m$} \rput(2.5,.3){$n$}
\pccurve[angleA=-90,angleB=-90,ncurv=1](-.8,-.2)(3.8,-.2)
\pccurve[angleA=90,angleB=180,ncurv=1](-.8,.2)(1.5,2.8)
\pccurve[angleA=90,angleB=0,ncurv=1](3.8,.2)(1.5,2.8)
\rput(1.5,3){$p_e$}
\psarc(3,1){.2}{0}{90}
\psline(3.2,1)(3.2,.2)
\psarc(2.6,2.25){.15}{270}{90}
\psline(2.6,2.4)(2.2,2.4)
\psline(2.2,1.8)(2.6,2.1)
\psline(2.6,1.8)(2.2,2.1)
\psarc(3,1){.2}{0}{90}
\psline(3.2,1)(3.2,.2)
\pccurve[angleA=-90,angleB=-90,ncurv=1](3.5,-.2)(-.5,-.2)
\pccurve[angleA=90,angleB=0,ncurv=1](3.5,.2)(3,1.5)
\psline(1.2,1.5)(1.8,1.5)
\pccurve[angleA=90,angleB=180,ncurv=1](-.5,.2)(.8,1.5)
\psline(2.6,1.5)(2.2,1.5)
\rput(1.5,-1.5){$p_i$}
\end{pspicture}
\\[3em]
+
\beta_{n+1}\beta_{m+1}&\begin{pspicture}[shift=-2](-1,-2)(4.5,2)
\psframe(-1,-.2)(0,.2)
\psframe(1,-.2)(2,.2)
\psframe(3,-.2)(4,.2)
\psframe(.8,1)(1.2,2)
\psframe(1.8,1)(2.2,2)
\psframe(0,1)(.4,2)
\psframe(2.6,1)(3,2)
 \psline(.8,2.1)(2.2,2.1)
 \pccurve[angleA=270,angleB=270,ncurv=1](-.2,-.2)(1.4,-.2)
 \psarc(1.2,1){.2}{0}{90}
 \psline(1.4,1)(1.4,.2)
 \psarc(1.8,1){.2}{90}{180}
 \psline(1.6,1)(1.6,.2)
 \pccurve[angleA=270,angleB=270,ncurv=1](1.6,-.2)(3.2,-.2)
 \rput(.5,.3){$m$} \rput(2.5,.3){$n$}
\pccurve[angleA=-90,angleB=-90,ncurv=1](-.8,-.2)(3.8,-.2)
\pccurve[angleA=90,angleB=180,ncurv=1](-.8,.2)(1.5,2.8)
\pccurve[angleA=90,angleB=0,ncurv=1](3.8,.2)(1.5,2.8)
\rput(1.5,3){$p_e$}
\psline(.4,1.2)(.8,1.2)
\psarc(0,1){.2}{90}{180}
\psline(-.2,1)(-.2,.2)
\psline(.4,2.1)(.8,1.8)
\psline(.8,2.1)(.4,1.8)
\psarc(.4,2.25){.15}{90}{270}
\psline(.4,2.4)(2.2,2.4)
\psarc(3,1){.2}{0}{90}
\psline(2.2,1.2)(2.6,1.2)
\psline(3.2,1)(3.2,.2)
\psline(2.2,1.8)(2.6,2.1)
\psline(2.6,1.8)(2.2,2.1)
\psarc(2.6,2.25){.15}{270}{90}
\psline(2.6,2.4)(2.2,2.4)
\pccurve[angleA=-90,angleB=-90,ncurv=1](3.5,-.2)(-.5,-.2)
\pccurve[angleA=90,angleB=0,ncurv=1](3.5,.2)(3,1.5)
\psline(1.2,1.5)(1.8,1.5)
\pccurve[angleA=90,angleB=180,ncurv=1](-.5,.2)(0,1.5)
\psline(2.6,1.5)(2.2,1.5)
\psline(.4,1.5)(.8,1.5)
\rput(1.5,-1.5){$p_i$}
\end{pspicture}
\hspace{-3.5 mm}
+
(\alpha_{n+1}\beta_{m+1}+\alpha_{m+1}\beta_{n+1})\begin{pspicture}[shift=-2](-1,-2)(4.5,2)
\psframe(-1,-.2)(0,.2)
\psframe(1,-.2)(2,.2)
\psframe(3,-.2)(4,.2)
\psframe(.8,1)(1.2,2)
\psframe(1.8,1)(2.2,2)
\psframe(0,1)(.4,2)
\psframe(2.6,1)(3,2)
 \psline(.8,2.1)(2.2,2.1)
 \pccurve[angleA=270,angleB=270,ncurv=1](-.2,-.2)(1.4,-.2)
 \psarc(1.2,1){.2}{0}{90}
 \psline(1.4,1)(1.4,.2)
 \psarc(1.8,1){.2}{90}{180}
 \psline(1.6,1)(1.6,.2)
 \pccurve[angleA=270,angleB=270,ncurv=1](1.6,-.2)(3.2,-.2)
 \rput(.5,.3){$m$} \rput(2.5,.3){$n$}
\pccurve[angleA=-90,angleB=-90,ncurv=1](-.8,-.2)(3.8,-.2)
\pccurve[angleA=90,angleB=180,ncurv=1](-.8,.2)(1.5,2.8)
\pccurve[angleA=90,angleB=0,ncurv=1](3.8,.2)(1.5,2.8)
\rput(1.5,3){$p_e$}
\psline(.4,1.2)(.8,1.2)
\psarc(0,1){.2}{90}{180}
\psline(-.2,1)(-.2,.2)
\psarc(.8,1.95){.15}{90}{270}
\psarc(.4,1.95){.15}{270}{90}
\psarc(.4,2.25){.15}{90}{270}
\psline(.4,2.4)(2.2,2.4)
\psarc(3,1){.2}{0}{90}
\psline(2.2,1.2)(2.6,1.2)
\psline(3.2,1)(3.2,.2)
\psline(2.2,1.8)(2.6,2.1)
\psline(2.2,2.1)(2.6,1.8)
\psarc(2.6,2.25){.15}{270}{90}
\psline(2.6,2.4)(2.2,2.4)
\pccurve[angleA=-90,angleB=-90,ncurv=1](3.5,-.2)(-.5,-.2)
\pccurve[angleA=90,angleB=0,ncurv=1](3.5,.2)(3,1.5)
\psline(1.2,1.5)(1.8,1.5)
\pccurve[angleA=90,angleB=180,ncurv=1](-.5,.2)(0,1.5)
\psline(2.6,1.5)(2.2,1.5)
\psline(.4,1.5)(.8,1.5)
\rput(1.5,-1.5){$p_i$}
\end{pspicture}
\end{flalign*}
\end{proof}

Putting the above formulas together, we can define the recursive evaluation:

\begin{align*} Net(m,n,p)=\left(A_p+\sum_{i=1}^{p-1}A_i \prod_{k=i+1}^p B_k\right)Net(m,n,p-1)
\end{align*}

and by direct calculation, we see that 
\begin{align*}
\prod_{i+1}^p B_k=\left(\frac{[2m+2i+2][m+i][n+i][2n+2i+2]}{[m+i+1][n+i+1]}\right)\left( \frac{[m+p+1][n+p+1]}{[2m+2p+2][2n+2p+2][n+p][m+p]}\right)
\end{align*}

simplifying the expression further to
\begin{align*}
A_p+\left( \frac{[m+p+1][n+p+1]}{[2m+2p+2][2n+2p+2][n+p][m+p]}\right)\left(\sum_{i=1}^{p-1}A_i\left(\frac{[2m+2i+2][m+i][n+i][2n+2i+2]}{[m+i+1][n+i+1]}\right)\right)
\end{align*}

when $q$ is a sufficiently large root of unity $(N=4\cdot(m+n+p+1)$, the left factor is always positive, as is the product, so it sufficient to check that $A_i$ is nonnegative for all $1 \leq i \leq p$. So, using the formulas

\begin{align*}
\alpha_n:=\frac{[2n][n+1][n-1]}{[2n+2][n]^2} && \beta_n:=\frac{[n-1]}{[n][2]},
\end{align*}

and

\begin{align*}
A_i:=-\frac{[6][2]}{[3]}+\alpha_{m+i}+\alpha_{n+i}+\frac{[6][2]}{[3]}(\beta_{n+i}+\beta_{m+i})-[4][2]\beta_{n+i}\cdot \beta_{m+i}&& B_i:=\alpha_{n+i} \cdot \alpha_{m+i}
\end{align*}

And by substitution, we see that

\begin{align*}A_i&=-\frac{[6][2]}{[3]}+\frac{[2(m+i)][m+i+1][m+i-1]}{[2m+2i+2][m+i]^2}+\frac{[2(n+i)][n+i+1][n+i-1]}{[2n+2i+2][n+i]^2}\\
&+\frac{[6][2]}{[3]}\left(\frac{[m+i-1]}{[m+i][2]}+\frac{[n+i-1]}{[n+i][2]} \right)-[4][2]\left(\frac{[m+i-1]}{[m+i][2]}\frac{[n+i-1]}{[n+i][2]}\right).
\end{align*}

\begin{theorem}
When $q$ is a root of unity of order greater than $2(a+b+c)+4$, we have that $\theta(a,b,c)\neq 0$.  
\end{theorem}
\begin{proof}
Making the substitution $q=e^{2\pi i/2(2k+6)}$, we let $N:=2(2k+6)$ and replace each quantum integer with $$[s]=\frac{\sin(2\pi s/N)}{\sin(2 \pi/N)}$$

and collecting terms in the denominator, and using the fact that the denominator is non-vanishing and positive, we see that it is sufficient to check

\begin{align*}A_i&=\sin(4s \cdot \pi/N)\sin((s+1)\cdot 2 \pi/N)\sin((s-1)\cdot 2 \pi/N)\sin((2j+2)\cdot 2 \pi/N)\sin^2(j\cdot 2 \pi/N)\\
&\cdot\sin(6 \pi/N) \sin(4 \pi/N) \sin(2 \pi/N)
+\sin(4j\cdot  \pi/N)\sin((j+1)\cdot 2 \pi/N)\sin((j-1)\cdot 2 \pi/N)\\
&\cdot \sin((2s+2)\cdot 2 \pi/N)\sin^2(s\cdot 2 \pi/N)\sin(6 \pi/N) \sin(4 \pi/N) \sin(2 \pi/N)+\sin((s-1)\cdot 2 \pi/N)\\
&\cdot\sin((2s+2)\cdot 2 \pi/N)\sin(k\cdot 2 \pi/N)\sin((2j+2)\cdot 2 \pi/N)\sin^2(j\cdot 2 \pi/N)\cdot\sin(12 \pi/N) \sin(4 \pi/N)\\ 
&\cdot\sin(2 \pi/N)+\sin((j-1)\cdot 2 \pi/N)\sin((2j+2)\cdot 2 \pi/N)\sin(j\cdot 2 \pi/N)\sin((2s+2)\cdot 2 \pi/N)\sin^2(s)\\
&\cdot\sin(12 \pi/N) \sin(4 \pi/N) \sin(2 \pi/N)-\sin((s-1)\cdot 2 \pi/N)\sin((j-1)\cdot 2 \pi/N)\sin((2s+2)\cdot 2 \pi/N)\\
&\cdot\sin((2j+2)\cdot 2 \pi/N)\sin(s\cdot 2 \pi/N)\sin(j\cdot 2 \pi/N)\sin(6 \pi/N)\sin(2 \pi/N)\sin(8\pi/N)\\
&-\sin(12 \pi/N)\sin^2(4 \cdot pi/N)\sin((2s+2)\cdot 2 \pi/N)\sin((2j+2)\cdot 2 \pi/N)\sin^2(j\cdot 2 \pi/N)\sin^2(s\cdot 2 \pi/N)
\end{align*}

where $s:=n+i$ while $j:=m+i$.

We claim that this is is nonzero for $N>4(m+n+i+1)$. This computation seems unwieldy, but is actually in a form that allows us to conclude our result.  To see this, one should note that the restriction that $N>4(m+n+i+1)$ gives that for every value of $x$ as appears above, $\sin(x) \in (0,\pi/2)$. This implies that each among the $\sin$ are monotonic. Then we see that the function is strictly negative, and thus along the discussion above we have that $\theta(a,b,c)$ is strictly nonzero.
\end{proof}
\begin{corollary}
In $\mathrm{Sp}(4)_k$ we have that $\mathrm{Hom}(a\otimes b\otimes c,\mathbb{C})$ is $1$ dimensional when $a+b+c$ is even, $a+b\geq c, a+c\geq b,$ and $b+c\geq a$ and if $a+b+c<2k+4$.
\end{corollary}
This provides a appropriate notion for admissibility for the $\mathrm{Sp}(4)_k$ link invariant using the $4$ dimensional representation.  We do note that this bound may not be sharp, meaning we don't have exactly when a triple is admissible, but only a sufficient condition.


\begin{thebibliography}{fszw90}
\bibitem{BB}
B.  Beliakova and C. Blanchet
\textit{Modular Categories of Type B,C,and D}. Comment. Math. Helv. 76 (2001) 467-500.
\bibitem{BHMV} C. Blanchet, N. Habegger, G. Masbaum, and P. Vogel, {\em Topological quantum field theories derived from the Kauffman Bracket}, Topology {\bf 34}(4) (1995) 883-927.
\bibitem{DK}
D. Kim,
\textit{Jones-Wenzl idempotents for Rank 2 Simple Lie Algebras}.
Osaka J. Math. 44(3) (2007), 691-722
\bibitem{GK}
G. Kuperberg,
\textit{Spiders for rank 2 Lie algebras}. 
Comm. Math. Phys. 180(1):109-151, 1996

\bibitem{J1} V.F.R. Jones, {\em Index for subfactors}.  Inventiones mathematicae 72.1 (1983): 1-25.

\bibitem{J2}
V.F.R Jones, {\em A polynomial invariant for knots via von Neumann algebras}. Bull. Amer. Math. Soc. (N.S.) 12 (1985), no. {\bf 1}, 103--111

\bibitem{K}
Kauffman, L. H. \textit{Knots and Physics.} Singapore: World Scientific, p. 33, 1991.

\bibitem{KL}
L.Kauffman, S. Lins
\textit{Temperley-Lieb Recoupling Theory and Invariants of 3-Manifolds},Princeton University Press, 1994, AMS-134

\bibitem{RT1} N. Reshetikhin and V. G. Turaev.  {\em Invariants of 3-manifolds via link polynomials and quantum
groups} Invent. Math. 103 (1991), no. {\bf 3}, 547597.

\bibitem{RT2} N. Reshetikhin and V. G. Turaev, {\em Ribbon graphs and their invariants derived from quantum
groups} Comm. Math. Phys.127 (1990), no. {\bf 1}, 126.
\bibitem{TW}
V. Turaev, H. Wenzl
\textit{Semisimple and Modular Categories from Link Invariants}, Math. Ann. 309, 411-461 (1997).

\end{thebibliography}
\end{document}